\documentclass[reqno,11pt]{amsart}

\usepackage{graphicx,subfigure}

\numberwithin{equation}{section}

\usepackage[latin1]{inputenc}
\usepackage[english]{babel}

\usepackage{amsmath,amsthm,amsfonts,latexsym,amssymb}
\usepackage[colorlinks]{hyperref}
\hypersetup{linkcolor=blue,citecolor=blue,filecolor=black,urlcolor=blue}
\usepackage{comment}

\usepackage{color,amsthm,amsfonts}
\definecolor{darkgreen}{rgb}{0,0.7,0.1}

{ \theoremstyle{plain}
\newtheorem{theorem}{Theorem}[section]
\newtheorem{proposition}[theorem]{Proposition}
\newtheorem{lemma}[theorem]{Lemma}
\newtheorem{corollary}[theorem]{Corollary}
  \theoremstyle{remark}
\newtheorem{remark}[theorem]{Remark}
  \theoremstyle{definition}
\newtheorem{definition}[theorem]{Definition}
}

\def\R{\mathbb{R}}

\def\N{\mathbb{N}}

\def\ga{`}
\def\ti{~}

\def\eps{\varepsilon}

\def\eps{\varepsilon}

\begin{document}


\title[A nonlocal supercritical Neumann problem]{A nonlocal supercritical Neumann problem}

\author[E. Cinti]{Eleonora Cinti}
\address{Dipartimento di Matematica\newline\indent
Alma Mater Studiorum Universit\`a di Bologna
\newline\indent
piazza di Porta San Donato, 5\newline\indent
40126 Bologna - Italy}
\email{eleonora.cinti5@unibo.it}

\author[F. Colasuonno]{Francesca Colasuonno}
\address{Dipartimento di Matematica ``Giuseppe Peano''\newline\indent
Universit\`a degli Studi di Torino
\newline\indent
via Carlo Alberto, 10\newline\indent
10123 Torino - Italy}
\email{francesca.colasuonno@unito.it}

\date{\today}

\begin{abstract}
We establish existence of positive non-decreasing radial solutions for a nonlocal nonlinear Neumann problem both in the ball and in the annulus. The nonlinearity that we consider is rather general, allowing  for supercritical growth (in the sense of Sobolev embedding). The consequent lack of compactness can be overcome, by working in the cone of non-negative and non-decreasing radial functions. Within this cone, we establish some a priori estimates which allow, via a truncation argument, to use variational methods for proving existence of solutions. As a side result, we prove a strong maximum principle for nonlocal Neumann problems, which is of independent interest.
\end{abstract}

\maketitle

\section{Introduction}
For $s>1/2$, we consider the following nonlocal Neumann problem
\begin{equation}\label{P}
\begin{cases}
(-\Delta)^s u+u=f(u)\quad&\mbox{in }\Omega,\\
u\ge 0\quad&\mbox{in }\Omega,\\
\mathcal N_s u=0\quad&\mbox{in }\mathbb R^n\setminus \overline \Omega.
\end{cases}
\end{equation}

Here $\Omega$ is a radial domain of $\mathbb R^n$, it is either a ball
\begin{equation}\label{palla}
\Omega=B_R:=\{x\in\mathbb R^n\,:\, |x|<R\}, \quad R>0,
\end{equation}
or an annulus 
\begin{equation}\label{anello}
\Omega=A_{R_0,R}:=\{x\in\mathbb R^n\,:\, R_0<|x|<R\}, \quad 0<R_0<R.
\end{equation}
Furthermore, $n\ge1$, $(-\Delta)^s$ denotes the fractional Laplacian
\begin{equation}\label{FL}
(-\Delta)^su(x):=c_{n,s}\,\mathrm{PV}\int_{\mathbb R^n}\frac{u(x)-u(y)}{|x-y|^{n+2s}}dy,
\end{equation}
and $\mathcal N_s$ is the following nonlocal normal derivative
\begin{equation}\label{Neu}
\mathcal N_s u(x):=c_{n,s}\int_\Omega\frac{u(x)-u(y)}{|x-y|^{n+2s}}dy\quad\mbox{for all }x\in\mathbb R^n\setminus\overline \Omega
\end{equation}
first introduced in \cite{DROV}, and $c_{n,s}$ is a normalization constant. It is a well known fact that the fractional Laplacian $(-\Delta)^s$ is the infinitesimal generator of a L\'evy process. The notion of nonlocal normal derivative $\mathcal N_s$ has also a particular probabilistic interpretation; we will comment on it later on in Section 2.  
We stress here that, with this definition of nonlocal Neumann boundary conditions, problem \eqref{P} has a variational structure. 

In this paper, we study the existence of non-constant solutions of \eqref{P} for a superlinear nonlinearity $f$, which can possibly be supercritical in the sense of Sobolev embeddings. 

In order to state our main result, we introduce the hypotheses on $f$. We assume that $f\in C^{1,\gamma}([0,\infty))$, for some $\gamma>0$, satisfies the following conditions:

\begin{itemize}
\item[$(f_1)$] $f'(0)=\lim_{t\to 0^+}\frac{f(t)}{t}\in (-\infty,1)$;
\item[$(f_2)$] $\liminf_{t\to\infty}\frac{f(t)}{t}>1$;
\item[$(f_3)$] there exists a constant $u_0>0$ such that $f(u_0)=u_0$ and $f'(u_0)>\lambda_2^{+,\mathrm{r}}+1$,
\end{itemize}
where $\lambda_2^{+,\mathrm{r}}>0$ is the second radial increasing eigenvalue of the fractional Laplacian with (nonlocal) Neumann boundary conditions.

Clearly, as a consequence of $(f_1)$, we know that $f(0)=0$ and $f$ is below the line $t$ in a right neighborhood of $0$. The results of the paper continue to hold if we weaken $(f_1)$ as follows
\begin{itemize}
\item[$(f'_1)$] $f(0)=0$, $f'(0)\in(-\infty,1]$ and $f(t)<t$ in $(0,\bar t)$ for some $\bar t>0$.
\end{itemize} 

A prototype nonlinearity satisfying $(f_1)$ and $(f_2)$ is given by 
$$
f(t):=t^{q-1}-t^{r-1},\mbox{ with }2\le r<q.
$$
For $q$ large enough, the above function satisfies condition $(f_3)$ as well.

We observe that $(f_1)$ and $(f_2)$ are enough to prove the existence of a mountain pass-type solution. The additional hypothesis $(f_3)$ is needed to prove that such a solution is non-constant. In particular, the existence of a fixed point $u_0$ of $f$ is a consequence of $(f_1)$, $(f_2)$, and the regularity of $f$; moreover, in view of $\int_\Omega (-\Delta)^s u dx=0$ (cf. \eqref{int-Deltas} below), the fact that $f(t)-t$ must change sign at least once is a natural compatibility condition for the existence of solutions. 
\smallskip

Our main result can be stated as follows. 

\begin{theorem}\label{thm:main}
Let $s>1/2$ and $f\in C^{1,\gamma}([0,\infty))$, for some $\gamma>0$, satisfy assumptions $(f_1)$--$(f_3)$. Then there exists a non-constant, radial, radially non-decreasing solution of \eqref{P} which is of class $C^2$ and positive almost everywhere in $\Omega$. In addition, if $u_{0,1},\ldots,u_{0,N}$ are $N$ different positive constants satisfying $(f_3)$, then \eqref{P} admits $N$ different non-constant, radial, radially non-decreasing, a.e. positive solutions. 
\smallskip 

If $\Omega=A_{R_0,R}$, the same existence and multiplicity result holds also for non-constant, radial, radially non-increasing, a.e. positive $C^2$ solutions of \eqref{P}.
\end{theorem}

We stress here that the situation with Neumann boundary conditions is completely different from the case with Dirichlet boundary conditions. Indeed, as for the local case $s=1$, a Poho\v{z}aev-type identity implies nonexistence of solutions under Dirichlet boundary conditions for critical or supercritical nonlinearities, cf. \cite[Corollary 1.3]{R-OS}, while here, under Neumann boundary conditions, we can find solutions even in the supercritical regime.   
Moreover, the supercritical nature of the problem prevents {\it a priori} the use of variational methods to attack the problem. Indeed, the energy functional associated to \eqref{P} is not even well-defined in the natural space where we look for solutions, i.e., $H^s_{\Omega,0}$ (cf. Section \ref{sec2}). To overcome this issue, we follow essentially the strategy used in \cite{BNW,CN}.
Our starting point is to work in the cone of non-negative, radial, non-decreasing functions  
\begin{equation}\label{cone}
\mathcal C_+(\Omega):=\left\{u\in H^s_{\Omega,0}\,:\, \begin{aligned}& u\mbox{ is radial and }u\ge0\mbox{ in }\mathbb R^n,\\\,&u(r)\le u(s) \mbox{ for all }R_0\le r\le s\le R\end{aligned}\right\},
\end{equation}
where with abuse of notation we write $u(|x|):=u(x)$ and in order to treat simultaneously the two cases $\Omega=B_R$ and $\Omega=A_{R_0,R}$, we assimilate $B_R$ into the limit case $A_{0,R}$. This cone was introduced for the local case ($s=1$) by Serra and Tilli in \cite{SerraTilli2011}, it is convex and closed in the $H^s$-topology.
The idea of working with radial functions, suggested by the symmetry of the problem, is dictated by the necessity of gaining compactness. Indeed, restricting the problem to the space of radial $H^s$ functions ($H^s_\mathrm{rad}$) allows somehow to work in a 1-dimensional domain, where we have better embeddings than in higher dimension. Nevertheless, in the case of the ball, the energy functional is not well defined even in  
$H^s_\mathrm{rad}$, since the sole radial symmetry is not enough to prevent the existence of sequences of solutions exploding at the origin. This is the reason for the increasing monotonicity request in the cone $\mathcal C_+$, cf. \cite{CM} for similar arguments in more general domains. Indeed, we can prove that all solutions of \eqref{P} belonging to $\mathcal C_+$ are a priori bounded in $H^s_{\Omega,0}$ and in $L^\infty(\Omega)$.
When the domain does not contain the origin, i.e. in the case of the annulus $R_0>0$, the monotonicity request can be avoided and it is possible to work directly in the space $H^s_\mathrm{rad}$. 
Nonetheless, working in $H^s_{\mathrm{rad}}$  would allow to prove the existence of just one radial weak solution of the equation in \eqref{P} under Neumann boundary conditions, whose sign and monotonicity are not known. Therefore, also in the case of the annulus, even if we do not need to gain compactness, we will work in $\mathcal C_+(A_{R_0,R})$ to find a non-decreasing solution, and in 
\begin{equation}\label{cone-}
\mathcal C_-(A_{R_0,R}):=\left\{u\in H^s_{A_{R_0,R},0}\,:\, \begin{aligned}&u\mbox{ is radial and }u\ge0\mbox{ in }\mathbb R^n,\\&u(r)\ge u(s) \mbox{ for all }R_0\le r\le s\le R\end{aligned}\right\},
\end{equation}
to find a non-increasing solution. 

For simplicity of notation, in the rest of the paper we will simply denote by $\mathcal C$ both $\mathcal C_+(\Omega)$ and $\mathcal C_-(A_{R_0,R})$, when the reasoning will be independent of the particular cone. 
  
In both cases, thanks to the a priori estimates, we can modify $f$ at infinity in such a way to obtain a  subcritical nonlinearity $\tilde f$.
This leads us to study a new {\it subcritical} problem, with the property that all solutions of the new problem belonging to $\mathcal C$ solve also the original problem \eqref{P}. The energy functional associated to the new problem is clearly well-defined in the whole $H^s_{\Omega,0}$. To get a solution of the new problem belonging to $\mathcal C$, we prove that a mountain pass-type theorem holds inside the cone $\mathcal C$. The main difficulty here is that we need to find a critical point of the energy, belonging to a set ($\mathcal C$) which is strictly smaller than the domain ($H^s_{\Omega,0}$) of the energy functional itself. To overcome this difficulty we build a deformation $\eta$ for the Deformation Lemma \ref{deformation} which preserves the cone, cf. also Lemma~\ref{cononelcono}. 
Once the minimax solution is found, we need to prove that it is non-constant. We further restrict our cone, working in a subset of $\mathcal C$ in which the only constant solution of \eqref{P} is the constant $u_0$ defined in $(f_3)$. In this set, we are able to distinguish the mountain pass solution from the constant using an energy estimate.

The multiplicity part of Theorem \ref{thm:main} can be easily obtained by repeating the same arguments around each constant solution $u_0$: in case we have more than one $u_0$  satisfying $(f_3)$, for each $u_{0,i}$, we work in a subset of $\mathcal C$ made of functions $u$ whose image is contained in a neighborhood of $u_{0,i}$. This allows us to localize each mountain pass solution and to prove that to each $u_{0,i}$ corresponds a different solution of the problem.

The paper is organized as follows:
\begin{itemize}
\item In Section 2, we recall some basic properties of our nonlocal Neumann problem. In particular, we describe its variational structure and we establish a strong maximum principle;
\item In Section 3, we prove the a priori bounds, both in $L^\infty$ and in the right energy space, which will be crucial for our existence result;
\item Section 4 contains the Mountain Pass-type Theorem (Theorem \ref{mountainpass}) which establishes existence of a radial, non-negative, non-decreasing solution and whose main ingredient relies on a Deformation Lemma inside the cone $\mathcal C$ (see Lemma \ref{deformation});
\item Finally, in Section 5, we prove that the solution, found via Mountain Pass argument, is not constant.
\end{itemize}

\section{The notion of nonlocal normal derivative and the variational structure of the problem}\label{sec2.0}

In this section, we comment on the notion of nonlocal normal derivative $\mathcal N_s$ and we describe some structural properties of the nonlocal Neumann problem under consideration, with particular emphasis on its variational structure.

As mentioned in the Introduction, we use the following notion of nonlocal normal derivative: 
\begin{equation}\label{N}
\mathcal N_su(x):=c_{n,s}\int_\Omega\frac{u(x)-u(y)}{|x-y|^{n+2s}}\,dy,\quad x\in \R^n\setminus \overline \Omega.
\end{equation}

As well explained in \cite{DROV}, with this notion of normal derivative, problem \eqref{P} has a variational structure. We emphasize that the operator $(-\Delta)^s $ that we consider is the standard fractional Laplacian on $\R^n$ (notice that the integration in \eqref{FL} is taken on the whole $\R^n$) and not the regional one (where the integration is done only on $\Omega$). This choice will be reflected in the associated energy functional (see e.g. \eqref{energyh}). We note in passing that, in \cite{A}, it is shown that the fractional Laplacian $(-\Delta)^su$ under homogeneous nonlocal Neumann boundary conditions ($\mathcal N_s u=0$ in $\mathbb R^n\setminus\overline{\Omega}$) can be expressed as a regional operator with a kernel having logarithmic behavior at the boundary.
There are other possible notions of \ga\ga Neumann conditions'' for problems involving fractional powers of the Laplacian (depending on which type of operator one considers), which all recover the classical Neumann condition in the limit case $s\uparrow 1$. See Setion 7 in \cite{DROV} and reference therein, for a more precise discussion on possible different definitions.

The choice of the standard fractional Laplacian $(-\Delta)^s$ and of the corresponding normal derivative $\mathcal N_s$ has also a specific probabilistic interpretation, that is well described in Section 2 of \cite{DROV}. The idea is the following. Let us a consider a particle that moves randomly in $\R^n$ according to the following law: if the particle is located at a point $x\in \R^n$, it can jump at any other point $y \in \R^n$ with a probability that is proportional to $|x-y|^{-n-2s}$. It is well known that the probability density $u(x,t)$ that the particle is situated at the point $x$ at time $t$, solves the fractional heat equation $u_t + (-\Delta)^s u =0$. If now we replace the whole space $\R^n$ with a bounded domain $\Omega$, we need to specify what are the ``boundary conditions", that is what happens when the particle exits $\Omega$. The choice of the Neumann condition $\mathcal N_s u=0$ corresponds to the following situation: when the particle reaches a point $x \in \R^n \setminus \overline \Omega$, it may jump back at any point $y \in \Omega$ with a probability density that, again, is proportional to $|x-y|^{-n-2s}$. Just as a comparison, if in place of Neumann boundary conditions, one considers the more standard Dirichlet boundary conditions (that in this nonlocal setting, reads $u\equiv 0$ in $\R^n \setminus \overline\Omega$), this would correspond to killing the particle when it exits $\Omega$.
 
We pass now to describe some variational properties of our nonlocal Neumann problem.
Let us start with an integration by part formula that justify the choice of $\mathcal N_s u$. In what follows $\Omega^c$ will denote the complement of $\Omega$ in $\R^n$.
\begin{lemma}[Lemma 3.3 in \cite{DROV}]\label{by-parts}
Let $u$ and $v$ be bounded $C^2$ functions defined on $\R^n$. Then, the following formula holds
\begin{equation}\label{int-by-parts}
\begin{aligned}
\frac{c_{n,s}}{2}\iint_{\R^{2n}\setminus(\Omega^c)^2}\frac{(u(x)-u(y))(v(x)-v(y))}{|x-y|^{n+2s}}&\,dx\,dy\\
&\hspace{-2cm}=\int_\Omega v\,(-\Delta)^su\,dx + \int_{\Omega^c}v\,\mathcal N_s u\,dx.
\end{aligned}
\end{equation}
\end{lemma}
\begin{remark}
As a consequence of Lemma \ref{by-parts}, if $u\in C^2(\mathbb R^n)$ solves \eqref{P}, taking $v\equiv 1$ in  \eqref{int-by-parts}, we get 
\begin{equation}\label{int-Deltas}
\int_\Omega (-\Delta)^su\,dx=0.
\end{equation}
\end{remark}

We now introduce the functional space where the problem is set.
 Let $u,\,v:\mathbb R^n\to \mathbb R$ be measurable functions, we set 
\begin{equation}\label{seminorm}
[u]_{H^s_{\Omega,0}}:= \left(\frac{c_{n,s}}{2}\iint_{\mathbb R^{2n}\setminus(\Omega^c)^2}\frac{|u(x)-u(y)|^2}{|x-y|^{n+2s}}dx\,dy\right)^{1/2}
\end{equation}
and we define the space 
$$
H^s_{\Omega,0}:=\{u:\mathbb R^n\to \mathbb R,\,u \in L^2(\Omega)\, :\, [u]_{H^s_{\Omega,0}}<+\infty \}
$$
 equipped with the scalar product
$$
(u,v)_{H^s_{\Omega,0}}:=\int_\Omega uv dx+\frac{c_{n,s}}{2}\iint_{\mathbb R^{2n}\setminus(\Omega^c)^2}\frac{(u(x)-u(y))(v(x)-v(y))}{|x-y|^{n+2s}}dx\,dy,
$$
and with the induced norm
\begin{equation}\label{norm}
 \|u\|_{H^s_{\Omega,0}}:=\|u\|_{L^2(\Omega)}+[u]_{H^s_{\Omega,0}}.
\end{equation}

By \cite[Proposition 3.1]{DROV}, we know that $(H^s_{\Omega,0},(\cdot,\cdot)_{H^s_{\Omega,0}})$ is a Hilbert space. 

In this paper we will mainly work with the notion of weak solutions for problem \eqref{P}, which naturally belong to the energy space $H^s_{\Omega,0}$ but, at some point (more precisely, when we will apply a strong maximum principle -- see Proposition \ref{MaxPrinc}) we will need to consider also classical solutions. For this reason, let us recall under which condition the fractional Laplacian given by the expression \eqref{FL} is well defined. 
Let $\mathcal L_s$ denote the following set of functions:
\begin{equation}\label{L_s}
\mathcal L_s:=\left\{u:\R^n\to \R\,:\,\int_{\R^n}\frac{|u(x)|}{1+|x|^{n+2s}}\,dx <\infty\right\}.
\end{equation}

Let $\Omega$ be a bounded set in $\R^n$, $s>1/2$, and let $u\in \mathcal L_s$ be a $C^{1,2s+\varepsilon -1}$ function in $\Omega$ for some $\varepsilon>0$. Then $(-\Delta)^s u$  is continuous on $\Omega$ and its value is given by the integral in \eqref{FL} (see Proposition 2.4 in \cite{Silvestre}). 

In particular, the condition $u \in \mathcal L_s$ ensures integrability at infinity for the integral in \eqref{FL}. Moreover, if $u$ belongs to the energy space $H^s_{\Omega,0}$, then automatically it is in $\mathcal L_s$, according to the following result.

\begin{lemma}
Let $\Omega$ be a bounded set in $\R^n$. Then
$$ H^s_{\Omega,0} \subset \mathcal L_s.$$
\end{lemma}

\begin{proof}
We prove that if $u \in H^s_{\Omega,0}$, then it satisfies the integrability condition
\begin{equation}\label{L_2}
\int_{\R^n}\frac{|u(x)|^2}{1+|x|^{n+2s}}\,dx<\infty,
\end{equation}
which, in particular, implies that $u \in \mathcal L_s$, by using H\"older inequality and observing that $(1+|x|^{n+2s})^{-1}\in L^1(\R^n)$. 

Throughout this proof we denote by $C$ many different positive constants whose precise value is not important for the goal of the proof and may change from time to time.
Let $\Omega'$ be a compact set contained in $\Omega$. We have
\begin{equation}\label{chain1}
\begin{split}
\infty &>\int_\Omega \int_{\R^n }\frac{|u(x)-u(y)|^2}{|x-y|^{n+2s}}\,dx\,dy  \\
& \ge\int_\Omega \int_{\Omega}\frac{|u(x)-u(y)|^2}{|x-y|^{n+2s}}\,dx\,dy + \int_{\Omega'} \int_{\R^n \setminus \Omega}\frac{|u(x)-u(y)|^2}{|x-y|^{n+2s}}\,dx\,dy \\
& \ge \int_\Omega \int_{\Omega}\frac{|u(x)-u(y)|^2}{|x-y|^{n+2s}}\,dx\,dy \\
&\hspace{2em}+ \frac{1}{2}\int_{\Omega'} \int_{\R^n \setminus\Omega}\frac{|u(x)|^2}{|x-y|^{n+2s}}\,dx\,dy -\int_{\Omega'} \int_{\R^n \setminus\Omega}\frac{|u(y)|^2}{|x-y|^{n+2s}}\,dx\,dy,
\end{split}
\end{equation}
where, in the last estimate we have used that $|a-b|^2 \ge \frac{1}{2} a^2-b^2$ by Young inequality.

Since $u\in H^s_{\Omega, 0}$, clearly the first term on the r.h.s is finite. Moreover, using that for $x\in \R^n \setminus\Omega$ and $y \in \Omega'$ one has that $|x-y|\ge \omega$, for some $\omega>0$, and the integrability of the kernel at infinity, we have for every $y\in\Omega'$
$$
\int_{\R^n \setminus\Omega}\frac{1}{|x-y|^{n+2s}}\,dx\le C\int_{\omega}^\infty \tau^{n-1-(n+2s)} d\tau=\frac{C}{\omega^{2s}},
$$
where $C$ is independent of $y\in\Omega'$.
Hence, 

\begin{equation}\label{chain1ba}
\begin{split}
\int_{\Omega'} \int_{\R^n \setminus\Omega}\frac{|u(y)|^2}{|x-y|^{n+2s}}\,dx\,dy&\le \int_{\Omega'} |u(y)|^2\left(\int_{\R^n \setminus\Omega}\frac{1}{|x-y|^{n+2s}}\,dx\right)\,dy\\
&\le  \frac{C}{\omega^{2s}}\int_{\Omega'}|u(y)|^2 dy < \infty.
\end{split}
\end{equation}

Therefore, combining \eqref{chain1} with \eqref{chain1ba}, we deduce that
$$
\int_{\Omega'} \int_{\R^n \setminus\Omega}\frac{ |u(x)|^2}{|x-y|^{n+2s}}\,dx\,dy < \infty.
$$

Finally, since $\Omega$ (and thus $\Omega'$) is bounded,  we have that there exists some number $d$ depending only on $\Omega$ such that $|x-y|\le d+ |x|$ for every $x \in \R^n \setminus\Omega$ and $y \in \Omega'$, which implies that
$$
\int_{\Omega'} \int_{\R^n \setminus\Omega}\frac{ |u(x)|^2}{|x-y|^{n+2s}}\,dx\,dy\ge |\Omega'| \int_{\R^n \setminus\Omega}\frac{ |u(x)|^2}{(d+|x|)^{n+2s}}\,dx.$$

This last inequality, together with the fact that
$$\int_{\Omega}\frac{ |u(x)|^2}{(d+|x|)^{n+2s}}\,dx <\infty,$$
(since $u \in L^2(\Omega)$) concludes the proof.

\end{proof}

Since it will be useful later on, we introduce also some standard notation for fractional Sobolev spaces.
We set
\begin{equation}\label{H_s}[u]_{H^s(\Omega)}:=\left(\frac{c_{n,s}}{2}\iint_{\Omega^2}\frac{|u(x)-u(y)|^2}{|x-y|^{n+2s}}\,dx\,dy \right)^{\frac{1}{2}}.\end{equation}
We denote by $H^s(\Omega)$ the space
$$H^s(\Omega):=\left\{u\in L^2(\Omega)\,:\, [u]_{H^s(\Omega)} <\infty\right\},$$
equipped with the norm
$$\|u\|_{H^s(\Omega)}=\|u\|_{L^2(\Omega)}+[u]_{H^s(\Omega)}.$$

Notice that in the definition $[u]_{H^s(\Omega)}$ the double integral is taken over $\Omega\times \Omega$, which differs from the seminorm defined in \eqref{seminorm} related to the energy functional of our problem.

Since the following obvious inequality holds between the usual $H^s$-seminorm and the seminorm $[\,\cdot\,]_{H^s_{\Omega,0}}$ defined in \eqref{seminorm}:
$$
[u]_{H^s_{\Omega,0}}\ge [u]_{H^{s}(\Omega)}\,
$$ 
as an easy consequence of the fractional compact embedding $H^s(\Omega)\hookrightarrow\hookrightarrow L^q(\Omega)$ (see for example Section 7 in \cite{Adams} and remind that $H^s(\Omega)=W^{s,2}(\Omega)$), we have the following.\smallskip

\begin{proposition}\label{embedding}
The space $H^s_{\Omega,0}$ is compactly embedded in $L^q(\Omega)$ for every $q\in[1,2_s^*)$,  where
$$2_s^*:=\begin{cases}\frac{2n}{n-2s}\quad&\mbox{if }2s<n,\\
+\infty&\mbox{otherwise}\end{cases}$$
is the fractional Sobolev critical exponent.
\end{proposition}

Given $h\in L^2(\Omega)$, we consider now the following linear problem 
\begin{equation}\label{Plin}
\begin{cases}
(-\Delta)^s u+u=h\quad&\mbox{in }\Omega,\\
\mathcal N_s u=0&\mbox{in }\R^n\setminus \overline{\Omega}.
\end{cases}
\end{equation}

\begin{definition}
We say that a function $u\in H^s_{\Omega,0}$ is a weak solution of problem \eqref{Plin} if
$$\frac{c_{n,s}}{2}\iint_{\R^{2n}\setminus(\Omega^c)^2}\frac{(u(x)-u(y))(v(x)-v(y))}{|x-y|^{n+2s}}\,dx\,dy + \int_\Omega u\, v\,dx=\int_\Omega h\,v.$$
\end{definition}

With this definition one can easily see that weak solutions of problem \eqref{Plin} can be found as critical points of the following energy functional defined on the space $H^s_{\Omega,0}$, cf. \cite[Proposition 3.7]{DROV}:
\begin{equation}\label{energyh}
\mathcal E(u):=\frac{c_{n,s}}{4}\iint_{\R^{2n}\setminus (\Omega^c)^2}\frac{|u(x)-u(y)|^2}{|x-y|^{n+2s}}\,dx\,dy+\frac{1}{2}\int_\Omega u^2\,dx -\int_\Omega hu\,dx.
\end{equation}

We state now a strong maximum principle for the fractional Laplacian with nonlocal Neumann conditions. 

\begin{theorem}\label{MaxPrinc}
Let $u\in C^{1,2s+\varepsilon-1}(\Omega)\cap \mathcal L_s$ (for some $\varepsilon>0)$ satisfy 
$$
\begin{cases}
(-\Delta)^s u\ge 0 & \mbox{in } \Omega,\\
u \ge 0 & \mbox{in } \Omega, \\
\mathcal N_s u \ge 0 &\mbox{in } \R^n\setminus \overline\Omega.
\end{cases}
$$

Then, either $u>0$ or $u \equiv 0$ a.e. in $\Omega$. 
\end{theorem}
\begin{proof}
Assume that $u$ is not a.e. identically zero and let us show that $u>0$ a.e. in $\Omega$.
We argue by contradiction: suppose that the set in $\Omega$ on which $u$ vanishes has positive Lebesgue measure, and let call it $Z$, i.e.
$$Z:=\{x\in \Omega\,|\,u(x)=0\},\quad \mbox{and}\quad |Z|>0.$$

Let now $\bar x \in Z$.
Since $u$ satisfies $(-\Delta)^s u\ge 0$ in $\Omega$, using the definition of fractional Laplacian,  we have that

\[\begin{split}
\int_{\R^n \setminus \Omega}\frac{u(\bar x)-u(y)}{|\bar x-y|^{n+2s}}\,dy&\ge -\int_{\Omega}\frac{u(\bar x)-u(y)}{|\bar x-y|^{n+2s}}\,dy\\
&=\int_{\Omega}\frac{u(y)}{|\bar x-y|^{n+2s}}\,dy >0,\end{split}\]
where the last strict inequality comes from the fact that we are assuming that $u$ is stritly positive on a subset of $\Omega$ of positive Lebesgue measure (otherwise it would be $u\equiv 0$ a.e. in $\Omega$).

Integrating the above inequality on the set $Z$ and using that $|Z|>0$, we deduce that
\begin{equation}\label{PM1}
\int_Z\int_{\R^n \setminus \Omega}\frac{u(\bar x)-u(y)}{|\bar  x-y|^{n+2s}}\,dy\,d\bar  x >0.\end{equation}

On the other hand, using that $u\ge 0$ in $\Omega$, we have

\begin{equation}\label{PM2}
\begin{split}
c_{n,s}\int_Z\int_{\R^n \setminus \Omega}\frac{u(\bar x)-u(y)}{|\bar x-y|^{n+2s}}\,dy\,d\bar  x &\le c_{n,s}\int_\Omega\int_{\R^n \setminus \Omega}\frac{u(x)-u(y)}{|x-y|^{n+2s}}\,dy\,dx \\
&=-c_{n,s}\int_{\R^n \setminus \Omega} \int_\Omega \frac{u(y)-u(x)}{|x-y|^{n+2s}}\,dx\,dy\\
&=-\int_{\R^n \setminus \Omega}\mathcal N_su(y)\,dy \leq 0.
\end{split}
\end{equation}
This contradicts \eqref{PM1} and concludes the proof.
\end{proof}

\begin{remark} Arguing in the same way, it is easy to see that the above strong maximum principle holds true when adding a zero order term in the equation satisfied in $\Omega$ (that is considering solutions of $(-\Delta)^s u(x) + c(x) u(x) \ge 0$ in $\Omega$).
\end{remark}

We conclude this Section with two results of \cite{DROV}.
The first one gives a further justification of calling $\mathcal N_s$ a ``nonlocal normal derivative".

\begin{proposition}[Proposition 5.1 of \cite{DROV}]
Let $\Omega$ be any bounded Lipschitz domain of $\R^n$ and let $u$ and $v$ be $C^2$ functions with compact support in $\R^n$.

Then,
$$\lim_{s\rightarrow 1} \int_{\R^n\setminus \Omega} \mathcal N_su\, v\,dx=\int_{\partial \Omega}\partial_\nu u \, v\,dx,$$
where $\partial_\nu$ denotes the external normal derivative to $\partial \Omega$.
\end{proposition}

The last result that we recall from \cite{DROV}, describes the spectrum of the fractional Laplacian with zero Neumann boundary conditions.
\begin{theorem}[Theorem 3.11 in \cite{DROV}]
There exists a diverging sequence of non-negative values
$$0=\lambda_1<\lambda_2\le \lambda_3\le \dots,$$
and a sequence of functions $u_i:\R^n\rightarrow \R$ such that
\[\begin{cases}
(-\Delta)^su_i(x)=\lambda_i u_i(x) & \mbox{for any } x\in \Omega\\
\mathcal N_s u_i(x)=0 & \mbox{for any } x\in \R^n\setminus \overline \Omega.
\end{cases}\]
Moreover, the functions $u_i$ (restricted to $\Omega$) provide a complete orthogonal system in $L^2(\Omega)$.
\end{theorem}

\section{A priori bounds for monotone radial solutions}\label{sec2}
Without loss of generality, from now on we suppose that $f$ satisfies the further assumption
\begin{itemize}
\item[$(f_0)$] $f(t)\ge 0$ and $f'(t)\ge 0$ for every $t\in[0,\infty)$.
 \end{itemize}
If this is not the case, it is always possible to reduce problem \eqref{P} to an equivalent one having a non-negative and non-decreasing nonlinearity, cf. \cite[Lemma\ti 2.1]{CN}. 
\smallskip

We look for solutions to \eqref{P} in the cone $\mathcal C$ defined in \eqref{cone} and \eqref{cone-}.

It is easy to prove that $\mathcal C$ is a closed convex cone in $H^s_{\Omega,0}$, i.e., the following properties hold for all $u,\,v\in\mathcal C$ and $\lambda\ge0$:
\begin{itemize}
\item[(i)] $\lambda u\in \mathcal C$;
\item[(ii)] $u+v\in \mathcal C$;
\item[(iii)] if also $-u\in\mathcal C$, then $u\equiv0$;
\item[(iv)] $\mathcal C$ is closed in the $H^s$-topology.
\end{itemize}
We will use the above properties of $\mathcal C$ in Lemma \ref{deformation}.
\medskip

We state now an embedding result for radial functions belonging to fractional Sobolev spaces, which can be found in \cite{SSV} (see also \cite{BSY}).

\begin{lemma}
If $s>1/2$ and $0<\bar R<R$, there exists a positive constant $C_{\bar R}=C_{\bar R}(\bar R,n,s)$ such that 
\begin{equation}\label{embLinf}
\left\|u\right\|_{L^\infty(B_R\setminus B_{\bar R})}\le C_{\bar R}\|u\|_{H^s_{B_R\setminus B_{\bar R},0}}
\end{equation}
for all $u$ radial in $H^s_{B_R\setminus B_{\bar R},0}$.
\end{lemma}
\begin{proof} The proof is the same as in \cite[Lemma 4.3]{BSY}, we report it here for the sake of completeness. 
Let $\bar R<\rho< R$. Using that $u$ is radial, $s>1/2$, and the trace inequality for $H^s(B_\rho\setminus B_{\bar R})$ (see e.g. \cite[Section 3.3.3]{Triebel}), we have for every $x\in \partial B_\rho$
$$
\begin{aligned}
|u(x)|^2&=\frac{\rho^{1-n}}{n\omega_n}\int_{\partial B_\rho}u^2 d\mathcal H^{n-1}\\
&\le C\frac{\rho^{1-n}}{n\omega_n} \rho^{2s-1}\left\{[u]_{H^s(B_\rho\setminus B_{\bar R})}^2+\frac{1}{\rho^{2s}}\|u\|_{L^2(B_\rho\setminus B_{\bar R})}^2\right\},
\end{aligned}
$$
where $\omega_n$ is the volume of the unit sphere in $\mathbb R^n$ and $d\mathcal H^{n-1}$ denotes the $(n-1)$-dimensional Hausdorff measure. 
We immediately deduce that for every $x\in \partial B_\rho$
\begin{equation*}
\begin{aligned}
|u(x)| &\le \begin{cases}
C|x|^{-\frac{n-2s}{2}}\|u\|_{H^s(B_\rho\setminus B_{\bar R})}\;&\mbox{ if } \rho=|x| \ge 1,\medskip \\
C\frac{|x|^{-\frac{n-2s}{2}}}{\rho^s}\|u\|_{H^s(B_\rho\setminus B_{\bar R})}&\mbox{ if } \rho=|x| < 1
\end{cases}\\
&\le C|x|^{-\frac{n-2s}{2}}\left(1+\frac{1}{\rho^s}\right)\|u\|_{H^s(B_\rho\setminus B_{\bar R})}\\
&\le C\bar R^{-\frac{n-2s}{2}}(1+\bar R^{-s})\|u\|_{H^s(B_R\setminus B_{\bar R})}
\end{aligned}
\end{equation*}

Hence, we conclude that 
$$
\begin{aligned}
\left\|u\right\|_{L^\infty(B_R\setminus B_{\bar R})}&\le C\bar R^{-\frac{n-2s}{2}}(1+\bar R^{-s})\|u\|_{H^s(B_R\setminus B_{\bar R})}\\
&\le C\bar R^{-\frac{n-2s}{2}}(1+\bar R^{-s})\|u\|_{H^s_{B_R\setminus B_{\bar R},0}},
\end{aligned}
$$
which proves the statement, with $C_{\bar R}:=C\bar R^{-\frac{n-2s}{2}}(1+\bar R^{-s})$.
\end{proof}

As mentioned above, working in the cones $\mathcal C$ of non-negative, radial and monotone  functions has the advantage to have an a priori $L^\infty$ bound, according to the following lemma. In particular, from the proof of the next lemma it will be clear the role of the {\it non-decreasing} monotonicity in the case of the ball. 

\begin{lemma}\label{L-infty}
Let $s>1/2$ and $\Omega$ be the ball $B_R$ or the annulus $A_{R_0,R}$ as in \eqref{palla}, \eqref{anello}. There exists a constant $C=C(R,R_0,n,s)>0$ such that
\[ \|u\|_{L^{\infty}(\Omega)}\le C\|u\|_{H^s_{\Omega,0}} \quad \mbox{for all } u\in \mathcal C.\]
\end{lemma}

\begin{proof} {\it Case $\Omega=B_R$.} In this case, $\mathcal C=\mathcal C_+(B_R)$. Since $u$ is radial and non-decreasing, we have that $\|u\|_{L^\infty(\Omega)}=\|u\|_{L^\infty(B_R\setminus  B_{R/2})}$. Hence, the conclusion follows by \eqref{embLinf}, observing that here $\bar R=R/2>0$.

{\it Case $\Omega=A_{R_0,R}$.} In the annulus, the same proof as before works both for $u\in\mathcal C_+$ and for $u\in\mathcal C_-$. We observe that in this case the constant $C$ depends on $R_0$ (and not on $R$).
\end{proof}

Thanks to the previous lemma, it would be enough to restrict the energy functional to $\mathcal C$ to get $\mathcal C$-constrained critical points; this is the approach in \cite{SerraTilli2011}. Nonetheless, as well explained in \cite{SerraTilli2011}, the cone $\mathcal C$ has empty interior in the $H^s$-topology, as a consequence it does not contain enough test functions to guarantee that constrained critical points are indeed free critical points. In \cite{SerraTilli2011}, the authors prove {\it a posteriori} that the constrained critical point that they find is a weak solution of the problem. In the present paper, we follow a different strategy proposed in \cite{BNW}, which, moreover, allows to cover a wider class of nonlinearities. The technique used relies on the truncation method and, for it, we need to prove a priori estimates for the solutions of \eqref{P} belonging to $\mathcal C$. We start with introducing some more useful notation.
\smallskip

Fix $\delta,\, M>0$ such that 
 \begin{equation}\label{Mdelta}
 f(t)\ge(1+\delta)t\quad\mbox{for all }t\ge M.
 \end{equation}
The existence of $\delta,\, M>0$ follows by $(f_2)$.
We introduce the following set of functions  
\begin{equation}\label{FmdeltaM}
\mathfrak{F}_{M,\delta}:=\left\{g\in C([0,\infty))\,:\,g\ge 0,\quad g(t)\ge(1+\delta)t\mbox{ for all }t\ge M\right\}.
\end{equation}
We remark that $\mathfrak F_{M,\delta}$ depends on $f$ only through $\delta$ and $M$. In the remaining part of this section, we shall derive some a priori estimates which are uniform in $\mathfrak F_{M,\delta}$ and hence depend only on $M$ and $\delta$, and not on the specific function $g$ belonging to $\mathfrak F_{M,\delta}$. This will be useful in the rest of the paper, since we will deal with a truncated function.

We give now the definition of weak solution for a general nonlinear Neumann problem of the form 

\begin{equation}\label{eq}\begin{cases}
(-\Delta)^su +u=g(u) & \mbox{in } \Omega\\
u \ge 0 & \mbox{in } \Omega\\
\mathcal N_s u=0 & \mbox{in } \R^n\setminus \overline \Omega.
\end{cases}
\end{equation}
\begin{definition}\label{def-weak}
We say that a non-negative function $u\in H^s_{\Omega,0}$ is a weak solution of problem \eqref{eq} if for every $v\in H^s_{\Omega,0}$
$$\frac{c_{n,s}}{2}\iint_{\R^{2n}\setminus(\Omega^c)^2}\frac{(u(x)-u(y))(v(x)-v(y))}{|x-y|^{n+2s}}\,dx\,dy + \int_\Omega u\, v\,dx=\int_\Omega  g(u)\,v.$$
\end{definition}

The following Lemma gives an $L^1$ bound for solutions to \eqref{eq} with $g$ belonging to the class $\mathfrak F_{M,\delta}$.

\begin{lemma}\label{L-1}
Let $g$ be any function in $\mathfrak F_{M,\delta}$. 
Then, there exists a constant $K_1=K_1(R,n,M,\delta)>0$ such that any weak solution $u\in \mathcal C$ of \eqref{eq} 
satisfies
$$\|u\|_{L^1(\Omega)}\le K_1.$$
\end{lemma}
\begin{proof} 
Testing the notion of weak solution with $v\equiv 1$ and using that $g\in\mathfrak F_{M,\delta}$, we get
$$\int_\Omega u \, dx=\int_{\{u< M\}}g(u)\,dx + \int_{\{u\ge M\}}g(u)\,dx  \ge (1+\delta) \int_{\{u\ge M\}}u\, dx.$$
Hence,
$$M|\Omega| \ge  \int_{\{u< M\}}u\, dx \ge \delta \int_{\{u\ge M\}}u\, dx$$
and so
$$\int_\Omega u\,dx=\int_{\{u< M\}}u\,dx+\int_{\{u\ge M\}}u\,dx\le M|\Omega|\left(1+\frac1\delta\right)=:K_1.$$
\end{proof}

The following lemma gives a uniform a priori bound in $L^\infty$ for solutions belonging to the cone $\mathcal C$ of problems \eqref{eq}, with $g\in\mathfrak F_{M,\delta}$. 

\begin{lemma}\label{L-infty-unif}
There exist two positive constants $K_\infty=K_\infty(R_0,R,n,s,M,\delta)$ and $K_2=K_2(R_0,R,n,s,M,\delta)$, such that for any $u\in \mathcal C$ weak solution of problem \eqref{eq}, the following estimates hold:
\[\|u\|_{L^\infty(\Omega)}\leq K_\infty\quad \mbox{and}\quad \|u\|_{H^s_{\Omega,0}}\leq K_2.\]
\end{lemma}

\begin{proof}
Choosing again $v \equiv 1$ in the definition of weak solution, we have
\begin{equation}\label{identity}
\int_\Omega u\,dx=\int_\Omega g(u)\,dx.
\end{equation}
On the other hand,  testing the equation with $u$ itself and using Lemma \ref{L-infty}, we deduce
\begin{equation}\label{chain}
\begin{split}
\|u\|^2_{L^\infty(\Omega)}&\le C^2 \left(\iint_{\R^{2n}\setminus(\Omega^c)^2}\frac{|u(x)-u(y)|^2}{|x-y|^{n+2s}}\,dx\,dy +\int_\Omega u^2\,dx\right) \\
&=C^2\int_\Omega g(u)\,u\,dx\leq C^2\|u\|_{L^\infty(\Omega)}\int_\Omega g(u)\,dx.
\end{split}\end{equation}

Combining \eqref{identity} with the previous estimate, we conclude that
$$\|u\|_{L^\infty(\Omega)}\leq C^2 \|u\|_{L^1(\Omega)}\leq C^2K_1=:K_\infty,$$
where the last estimate comes from Lemma \ref{L-1}.
Finally, this bound on $\|u\|_{L^\infty(\Omega)}$ combined with inequality \eqref{chain} above, gives the following uniform bound on $\|u\|_{H^s_{\Omega,0}}$:
$$\|u\|_{H^s_{\Omega,0}}^2\le \|u\|_{L^\infty(\Omega)}\int_\Omega g(u)\,dx=\|u\|_{L^\infty(\Omega)}\|u\|_{L^1(\Omega)}\le C^2K_1^2=:K_2^2.$$
\end{proof}

We now prove a regularity result for weak solutions of \eqref{P} belonging to the cone $\mathcal C$.

\begin{lemma}\label{regularity}
Let $u\in \mathcal C$ be a weak solution of \eqref{P}. Then $u\in C^{2}(\Omega)\cap L^\infty(\mathbb R^n)$. 
\end{lemma}
\begin{proof}
By Lemma \ref{L-infty-unif} we know that $u\in L^\infty(\Omega)$. 
Furthermore, by the nonlocal Neumann boundary conditions, we have that 
$$
u(x)=\frac{\displaystyle{\int_\Omega}\frac{u(y)}{|x-y|^{n+2s}}dy}{\displaystyle{\int_\Omega}\frac{1}{|x-y|^{n+2s}}dy}\quad\mbox{for every }x\in \mathbb R^n\setminus\overline\Omega.
$$
Thus, for any $\varepsilon>0$ we get for every $x\in \mathbb R^n\setminus\Omega_\varepsilon:=\{x\in\mathbb R^n\,:\, \mathrm{dist}(x,\Omega)\ge\varepsilon\}$
$$
u(x)=\frac{\displaystyle{\int_\Omega}\frac{u(y)}{|x-y|^{n+2s}}dy}{\displaystyle{\int_\Omega}\frac{1}{|x-y|^{n+2s}}dy}\le \|u\|_{L^\infty(\Omega)}\frac{\displaystyle{\int_\Omega}\frac{1}{|x-y|^{n+2s}}dy}{\displaystyle{\int_\Omega}\frac{1}{|x-y|^{n+2s}}dy}=\|u\|_{L^\infty(\Omega)}.
$$
Therefore, being this estimate uniform in $\varepsilon$, we get $|u(x)|\le \|u\|_{L^\infty(\Omega)}$ for every $x\in \mathbb R^n\setminus\overline\Omega$. 
Hence, $u\in L^\infty(\mathbb R^n)$ and so, using \cite[Proposition 2.9]{Silvestre} with $w=f(u)-u\in L^\infty(\mathbb R^n)$, we obtain $u\in C^{1,\alpha}(\Omega)$ for every $\alpha\in (0,2s-1)$. Then, recalling that $f\in C^{1,\gamma}$, we can use a bootstrap argument, and apply \cite[Proposition 2.8]{Silvestre} to conclude the proof. 
\end{proof}

\section{Existence of a mountain pass radial solution}\label{sec:mountain_pass_existece}
In this section we prove the existence of a radial solution of \eqref{P} via a Mountain Pass-type Theorem. 
We are now ready to start the truncation method described in the Introduction: we will modify $f$ in $(K_\infty,+\infty)$, where $K_\infty$ is the $L^\infty$ bound given in Lemma~\ref{L-infty-unif}, in such a way to have a subcritical nonlinearity $\tilde f$.\smallskip

\begin{lemma}\label{truncated} For every $\ell\in(2,2_s^*)$, there exists $\tilde{f}\in \mathfrak F_{M,\delta}\cap C^1([0,\infty))$, satisfying $(f_0)$--$(f_3)$, 
 \begin{equation}\label{subcritical}
\lim_{t\to\infty}\frac{\tilde{f}(t)}{t^{\ell-1}}=1,
\end{equation}
and with the property that if $u\in\mathcal C$ solves 
\begin{equation}\label{tildeP}\begin{cases}(-\Delta)^s u+u=\tilde{f}(u)\quad&\mbox{in }\Omega,\\
u>0&\mbox{in }\Omega,\\
\mathcal N_s u=0&\mbox{in }\R^n\setminus \overline \Omega,
\end{cases}
\end{equation}
then $u$ solves \eqref{P}.  
\end{lemma}
For the proof of the above lemma, we refer the reader to \cite[Lemma 4.3]{BNW}.

As a consequence of the previous lemma, condition $(f_1)$, and the regularity of $f$, there exists $C>0$ for which 
\begin{equation}\label{crescitasottocritica}
\tilde f(t)\le C (1+t^{\ell-1})\quad\mbox{for all }t\ge0,
\end{equation}
where $\ell\in(2,2_s^*)$.

From now on in the paper, we consider the trivial extension of $\tilde f$, still denoted with the same symbol 
$$\tilde f= \begin{cases} \tilde f\quad&\mbox{in }[0,+\infty),\\
0&\mbox{in }(-\infty, 0).
\end{cases}$$

Recalling the Definition \ref{def-weak} of weak solution (applied here with $g=\tilde f$)  one can easily see that weak solutions of problem \eqref{tildeP} can be found as critical points of the following energy functional defined on the space $H^s_{\Omega,0}$:
\begin{equation}\label{energy}
\mathcal E(u):=\frac{c_{n,s}}{4}\iint_{\R^{2n}\setminus (\Omega^c)^2}\frac{|u(x)-u(y)|^2}{|x-y|^{n+2s}}\,dx\,dy +\frac{1}{2}\int_\Omega u^2\,dx -\int_\Omega \tilde F(u)\,dx,
\end{equation}
where $\tilde F(t):=\int_0^t \tilde f(\tau)d\tau.$
The proof of this fact follows from the argument in the proof of Proposition 3.7 in \cite{DROV}, with the obvious modifications due the presence of the nonlinearity $\tilde f$.

Because of \eqref{subcritical} and the Sobolev embedding, the functional $\mathcal E$ is well defined and of class $C^2$, being $s>1/2$.  

\begin{lemma}[\textbf{Palais-Smale condition}]\label{PalaisSmale} The functional $\mathcal E$ satisfies the Palais-Smale condition, i.e. every \sl{(PS)}-sequence $(u_k)\subset H^s_{\Omega,0}$, namely a sequence satisfying
$$(\mathcal E(u_k)) \mbox{ is bounded}\quad\mbox{ and }\quad \mathcal E'(u_k)\to 0\mbox{ in }(H^s_{\Omega,0})^*,$$ 
admits a convergent subsequence.
\end{lemma}
\begin{proof}
Reasoning as in \cite[Lemma 3.3]{CN}, as a consequence of \eqref{subcritical},
there exist $\mu\in(2,\ell]$ and $T_0>0$ such that 
\begin{equation}\label{AmbRa}
\tilde{f}(t)t\ge\mu\tilde{F}(t)\quad\mbox{ for all }t\ge T_0.
\end{equation}
Now, let $(u_k)\subset H^s_{\Omega,0}$ be a (PS)-sequence for $\mathcal E$ as in the statement. We estimate
$$
\begin{aligned}
\mathcal E(u_k)-\frac1\mu \mathcal E'(u_k)[u_k]\ge &\frac{c_{n,s}}{2} \left(\frac12-\frac1\mu\right)\|u_k\|_{H^s_{\Omega,0}}^2\\
&+\int_{\{u_k\le T_0\}}\left(\frac1\mu\tilde{f}(u_k)u_k-\tilde F(u_k)\right)dx
\end{aligned}
$$
and, being $(u_k)$ a (PS)-sequence,  
$$
\mathcal E(u_k)-\frac1\mu \mathcal E'(u_k)[u_k]\le |\mathcal E(u_k)|+\frac1\mu \|\mathcal E'(u_k)\|_*\|u_k\|_{H^s_{\Omega,0}}\le C(1+\|u_k\|_{H^s_{\Omega,0}})
$$
for some $C>0$, where we have denoted by $\|\cdot\|_*$ the norm of the dual space of $H^s_{\Omega,0}$. Since we know that $\int_{\{u_k\le T_0\}}\left(\frac1\mu\tilde{f}(u_k)u_k-\tilde F(u_k)\right)dx$ is uniformly bounded in $k$, we get 
$$\left(\frac12-\frac1\mu\right)\|u_k\|_{H^s_{\Omega,0}}^2\le C(1+\|u_k\|_{H^s_{\Omega,0}}).$$
Therefore, $(u_k)$ is bounded in $H^s_{\Omega,0}$ and so there exists $u\in H^s_{\Omega,0}$ such that $u_k\rightharpoonup u$ in $H^s_{\Omega,0}$, up to a subsequence. By compact embedding (Proposition\ti\ref{embedding}), $u_k\to u$ in $L^\ell(\Omega)$ and, up to a subsequence, $u_k\to u$ a.e. in $\Omega$.
Again, since $(u_k)$ is a (PS)-sequence
\begin{equation}\label{E'to0}
|\mathcal E'(u_k)[u_k-u]|\le \|\mathcal E'(u_k)\|_{*}\|u_k-u\|_{H^s_{\Omega,0}}\to 0\quad \mbox{as }k\to\infty.
\end{equation}
On the other hand, by H\"older's inequality and \eqref{crescitasottocritica},
\begin{equation}\label{int-ftilde}
\begin{aligned}
\int_\Omega \tilde f(u_k)(u_k-u)\,dx &\le C\int_\Omega(1+u_k^{\ell-1})(u_k-u)dx\\
&\le C\|1+u_k\|^{\ell-1}_{L^\ell(\Omega)}\|u_k-u\|_{L^\ell(\Omega)}\to 0\quad\mbox{as }k\to\infty
\end{aligned}
\end{equation} 
and 
\begin{equation}\label{int-uk-u}
\int_\Omega u_k(u_k-u)\,dx=\int_\Omega (u_k-u)^2\,dx+\int_\Omega u(u_k-u)\,dx\to0 \quad\mbox{as }k\to\infty.
\end{equation}
Recalling that  
$$
\begin{aligned}
\mathcal E'(u_k)&[u_k-u]\\
&=\frac{c_{n,s}}{2}\iint_{\R^{2n}\setminus(\Omega^c)^2}\frac{(u_k(x)-u_k(y))[(u_k-u)(x)-(u_k-u)(y)]}{|x-y|^{n+2s}}\,dx\,dy\\
&\phantom{=}+\int_\Omega u_k(u_k-u)\,dx -\int_\Omega \tilde f(u_k)(u_k-u)\,dx,
\end{aligned}
$$
by \eqref{E'to0}, we have in view of \eqref{int-ftilde} and \eqref{int-uk-u} 
\begin{equation}\label{convergence0}
\lim_{k\to\infty}\iint_{\R^{2n}\setminus(\Omega^c)^2}\frac{(u_k(x)-u_k(y))[(u_k-u)(x)-(u_k-u)(y)]}{|x-y|^{n+2s}}\,dx\,dy=0.
\end{equation}
We claim that \eqref{convergence0} implies the following
\begin{equation}\label{conv-norms}
\lim_{k\to\infty}\iint_{\R^{2n}\setminus(\Omega^c)^2}\frac{|u_k(x)-u_k(y)|^2}{|x-y|^{n+2s}}\,dx\,dy= \iint_{\R^{2n}\setminus(\Omega^c)^2}\frac{|u(x)-u(y)|^2}{|x-y|^{n+2s}}\,dx\,dy.
\end{equation}
Indeed, by weak lower semicontinuity 
\begin{equation}\label{wlsc}
\iint_{\R^{2n}\setminus(\Omega^c)^2}\frac{|u(x)-u(y)|^2}{|x-y|^{n+2s}}\,dx\,dy\le \liminf_{k\to\infty}\iint_{\R^{2n}\setminus(\Omega^c)^2}\frac{|u_k(x)-u_k(y)|^2}{|x-y|^{n+2s}}\,dx\,dy.
\end{equation}
Moreover, setting
$$a:=u(x)-u(y)\quad\mbox{and}\quad b:= u_k(x)-u_k(y),$$
using the easy inequality $a^2+2b(b-a)\ge b^2$, we deduce
$$
\begin{aligned}
&\iint_{\R^{2n}\setminus(\Omega^c)^2}\frac{|u(x)-u(y)|^2}{|x-y|^{n+2s}}\,dx\,dy\\
&\hspace{1cm}+2\iint_{\R^{2n}\setminus(\Omega^c)^2}\frac{(u_k(x)-u_k(y))(u_k(x)-u_k(y)-u(x)+u(y))}{|x-y|^{n+2s}}\,dx\,dy\\
&\hspace{1cm}\ge \iint_{\R^{2n}\setminus(\Omega^c)^2}\frac{|u_k(x)-u_k(y)|^2}{|x-y|^{n+2s}}\,dx\,dy.
\end{aligned}
$$
Thus, by \eqref{convergence0}, we obtain
$$
\iint_{\R^{2n}\setminus(\Omega^c)^2}\frac{|u(x)-u(y)|^2}{|x-y|^{n+2s}}\,dx\,dy\ge \limsup_{k\to\infty}\iint_{\R^{2n}\setminus(\Omega^c)^2}\frac{|u_k(x)-u_k(y)|^2}{|x-y|^{n+2s}}\,dx\,dy,
$$
which, together with \eqref{wlsc}, proves the claim.
Combining \eqref{conv-norms} with the convergence of $L^2$ norms $\|u_k\|_{L^2(\Omega)}^2\to\|u\|_{L^2(\Omega)}^2$, we get 
$$
\|u_k\|_{H^s_{\Omega,0}}\to\|u\|_{H^s_{\Omega,0}}.
$$
Finally, since we also have weak convergence $u_k\rightharpoonup u$ in $H^s_{\Omega,0}$, we conclude that $u_k\to u$ in $H^s_{\Omega,0}$.
\end{proof}

\begin{remark} We observe that, as already noticed in \cite[Remark 4.13]{BNW}, the truncation method (cf. Lemma \ref{truncated}) and the preliminary a priori estimate (cf. Lemma \ref{L-infty-unif}) are needed to get the subcritical growth of the nonlinearity \eqref{crescitasottocritica} and the Ambrosetti-Rabinowitz condition \eqref{AmbRa}. If the original nonlinearity $f$ of problem \eqref{P} satisfies those further assumptions, it is possible to skip the first part concerning a priori estimates and truncation, and to prove directly the existence of both a non-decreasing and a non-increasing (also for the ball) solutions, just starting from Lemma \ref{PalaisSmale} with $\tilde f=f$. 
\end{remark}

We define 
\begin{equation}\label{u-u+}
\begin{aligned}
u_-&:= \sup \{t \in [0,u_0)\,:\, \tilde f(t)=t\},\\
u_+&:= \inf \{t \in (u_0,+\infty) \,:\, \tilde f(t)=t\}.
\end{aligned}
\end{equation}
Since $\tilde f$ is a truncation of $f$, using Lemma \ref{truncated} and the properties satisfied by $f$, we have that $\tilde f(u_0)=u_0$ and $\tilde f'(u_0)>0$, so that $u_0$ is an isolated 
zero of the function $\tilde f(t)-t$. Hence, 
\begin{equation}\label{u-+}
u_-\neq u_0\quad\mbox{ and }\quad u_+\neq u_0.
\end{equation} 
We point out that $u_+= +\infty$ is possible. Next, in order to localize the solutions, as already explained in the Introduction, we define the restricted cones 
\begin{equation*}
\begin{aligned}
\mathcal C_{+,*}&:= \{u \in \mathcal C_+\::\: \text{$u_- \le  u \le u_+$ in $\Omega$}\},\\
\mathcal C_{-,*}&:= \{u \in \mathcal C_-\::\: \text{$u_- \le  u \le u_+$ in $A_{R_0,R}$}\}.
\end{aligned}
\end{equation*} 
As for $\mathcal C$, when it will not be relevant to distinguish between the two cones $\mathcal C_{+,*}$ and $\mathcal C_{-,*}$, we will simply denote by $\mathcal C_*$ either of them
\begin{equation}\label{Cstar}
\mathcal C_*:= \{u \in \mathcal C\::\: \text{$u_- \le  u \le u_+$ in $\Omega$}\}.
\end{equation}

Clearly, $\mathcal C_*$ is closed and convex. 

\begin{corollary}\label{conseqPS} Let $c\in\mathbb R$ be such that $\mathcal E'(u)\neq 0$ for all $u\in \mathcal C_*$ with $\mathcal E(u)=c$. Then, there exist two positive constants $\bar\varepsilon$ and $\bar\delta$ such that the following inequality holds  
$$
\|\mathcal E'(u)\|_*\ge\bar\delta\quad	\mbox{for all }u\in \mathcal C_*\mbox{ with }|\mathcal E(u)-c|\le 2\bar\varepsilon.
$$
\end{corollary}
\begin{proof} The proof follows by Lemma \ref{PalaisSmale}. Indeed, suppose by contradiction that the thesis does not hold, then we can find a sequence $(u_k)\subset \mathcal C_*$ such that $\|\mathcal E'(u_k)\|_*<\frac1k$ and $c-\frac1k \le \mathcal E(u_k)\le c+\frac1k$ for all $k$. Hence, $(u_k)$ is a Palais-Smale sequence, and  since $\mathcal E$ satisfies the Palais-Smale condition, up to a subsequence, $u_k\to u$ in $H^s_{\Omega,0}$. Since $(u_k)\subset \mathcal C_*$ and $\mathcal C_*$ is closed, $u\in \mathcal C_*$. The fact that $\mathcal E$ is of class $C^1$ then gives $\mathcal E(u_k)\to c=\mathcal E(u)$ and $\mathcal E'(u_k)\to 0=\mathcal E'(u)$, which contradicts the hypothesis.
\end{proof}

We define the operator $T:(H^s_{\Omega,0})^*\to H^s_{\Omega,0}$ as 
\begin{equation}\label{T}T(h)=v, \quad\mbox{where $v$ solves }
\quad(P_h)\;\begin{cases}(-\Delta)^s v+v=h\quad&\mbox{in }\Omega,\\
\mathcal N_s v=0&\mbox{in }\R\setminus\overline \Omega.
\end{cases}
\end{equation}
The associated energy of $(P_h)$, given by \eqref{energyh}, is strictly convex, coercive and weakly lower semicontinuous, hence problem $(P_h)$ admits a unique weak solution $v\in H^s_{\Omega,0}$, which is a minimizer of the energy. Hence, the definition of $T$ is well posed and 
\begin{equation}\label{eq:T_continuous}
T\in C((H^s_{\Omega,0})^*;H^s_{\Omega,0}),
\end{equation}
(see for instance the proof Theorem 3.9 in \cite{DROV}).

We introduce also the operator
\begin{equation}\label{eq:tildeT_def}
\tilde T:H^s_{\Omega,0}\to H^s_{\Omega,0} \quad\textrm{defined by}\quad \tilde T(u)=T(\tilde f(u)),
\end{equation}
with $T$ given in \eqref{T}. Being $\ell<2_s^*$, $u\in H^s_{\Omega,0}$ implies $u\in L^\ell(\Omega)$. Hence, by \eqref{crescitasottocritica}, $\tilde f(u)\in L^{\ell'}(\Omega)\subset (H^s_{\Omega,0})^*$ and $\tilde T$ is well defined. 

\begin{proposition}\label{Ttildecompact}
The operator $\tilde T$ is compact, i.e. it maps bounded subsets of $H^s_{\Omega,0}$ into precompact subsets of $H^s_{\Omega,0}$. 
\end{proposition}
\noindent The proof of the previous proposition is the same as for \cite[Proposition 3.2]{CN} with the obvious changes due to the different space we are working in, so we omit it. \smallskip

In the following lemma we prove that the operator $\tilde T$ preserves the cone $\mathcal C_*$, which in turn will be useful, in Lemma \ref{deformation}, to build a deformation that preserves the cone. As mentioned in the Introduction, this is crucial to guarantee existence of a minimax solution in $\mathcal C_*$.

\begin{lemma}\label{cononelcono}
The operator $\tilde T$ defined in \eqref{eq:tildeT_def} satisfies
$\tilde T(\mathcal C_*)\subseteq \mathcal C_*$.
\end{lemma}
\begin{proof} We first note that $u\in\mathcal C_*$ implies $\tilde f(u)\in\mathcal C$, by the properties of $\tilde f$. Now, let $u\in\mathcal C_*$ and $v:=\tilde T(u)$. 
We see that $v\ge0$ in $\Omega$. Indeed, denoting by $v^+$ the positive part of $v$, by an easy observation we have that $|v^+(x)-v^+(y)|\le |v(x)-v(y)|$, and hence $\mathcal E(v^+)\le \mathcal E(v)$. Furthermore, due to uniqueness, $v$ is radial. 
For the monotonicity, we distinguish the two cases. 

{\it Case $u\in \mathcal C_{+,*}.$} In this case, we have to prove that $v$ is non-decreasing. It is enough to show that for every $r\in(R_0,R)$ one of the following cases occurs:
\begin{itemize}
\item[$(a)$] $v(t)\le v(r)$ for all $t\in(R_0,r)$, 
\item[$(b)$] $v(t)\ge v(r)$ for all $t\in(r,R)$. 
\end{itemize}
Indeed, if $v(\bar t)>v(r)$ for some $R_0<\bar t<r$, by the continuity of $v$, there exists $t\in(\bar t,r)$ for which $v(\bar t)>v(t)> v(r)$ which violates both $(a)$ and $(b)$.   
Now, we fix $r\in(R_0,R)$. If $\tilde f(u(r))\le v(r)$, we consider the test function
$$
\varphi_+(x):=\begin{cases}(v(|x|)-v(r))^+\quad&\mbox{if }R_0<|x|\le r,\\
0&\mbox{otherwise}. 
\end{cases}
$$

We have
\begin{equation}\label{4.16}
\begin{split}
&\iint_{\R^{2n}\setminus \left((B_r\setminus B_{R_0})^c\right)^2} \frac{(v(x)-v(y))(\varphi_+(x)-\varphi_+(y))}{|x-y|^{n+2s}}dx\,dy \\
&\hspace{5.6cm} +\int_{B_r\setminus B_{R_0}}v(x)\varphi_+(x)dx \\
&\hspace{1em} =\int_{B_r\setminus B_{R_0}}\tilde f(u(x))\varphi_+(x)dx \le \tilde f(u(r))\int_{B_r\setminus B_{R_0}}\varphi_+(x)dx\\
&\hspace{1em} \le v(r)\int_{B_r\setminus B_{R_0}}\varphi_+(x)dx.
\end{split}
\end{equation}

Using again the definition of $\varphi_+$, we obtain 
\begin{equation}\label{4.17}
\begin{aligned}
\iint_{\R^{2n}\setminus \left((B_r\setminus B_{R_0})^c\right)^2} &\frac{(v(x)-v(y))(\varphi_+(x)-\varphi_+(y))}{|x-y|^{n+2s}}dx\,dy
\\
&\ge \iint_{\R^{2n}\setminus \left((B_r\setminus B_{R_0})^c\right)^2} \frac{|\varphi_+(x)-\varphi_+(y)|^2}{|x-y|^{n+2s}}dx\,dy.
\end{aligned}
\end{equation}
Hence, by \eqref{4.16} and \eqref{4.17}
\begin{equation*}
\begin{split}
0&\ge\\
 &\iint_{\R^{2n}\setminus \left((B_r\setminus B_{R_0})^c\right)^2} \frac{|\varphi_+(x)-\varphi_+(y)|^2}{|x-y|^{n+2s}}dx\,dy + \int_{B_r\setminus B_{R_0}}(v(x)-v(r))\varphi_+(x)dx\\
&=\iint_{\R^{2n}\setminus \left((B_r\setminus B_{R_0})^c\right)^2} \frac{|\varphi_+(x)-\varphi_+(y)|^2}{|x-y|^{n+2s}}dx\,dy + \int_{B_r\setminus B_{R_0}}|\varphi_+(x)|^2dx,
\end{split}
\end{equation*}
which gives $\varphi_+ \equiv 0$, i.e. (a) holds.

Analogously, if $\tilde f(u(r)) > v(r)$, we consider the test function
$$
\varphi_-(x):=\begin{cases}0\quad&\mbox{if }R_0<|x|\le r,\\
(v(|x|)-v(r))^-&\mbox{otherwise} 
\end{cases}$$ 
and we prove that $(b)$ holds. Therefore, we have proved that $v$ is nondecreasing. 

{\it Case $u\in \mathcal C_{-,*}.$} In this case, we know that $u\in \mathcal C_{-,*}$ and  have to prove that $v$ is non-increasing. The proof is the same as for $\mathcal C_{+,*}$ changing the roles of  $\varphi_+$ and $\varphi_-$.
\smallskip

It remains to show that $u_-\le v\le u_+$. By the fact that $\tilde{f}(u_-)= u_-$ and that $\tilde f$ is non-decreasing we get 
$$(-\Delta)^s (v-u_-)+(v-u_-)=\tilde f(u)-\tilde f(u_-)\ge0.$$
Multiplying the equation above by $(v-u_-)^-$, integrating it over $\Omega$, and using that for any $g$ one has that
$-|g^-(x)-g^-(y)|^2\ge (g(x)-g(y))(g^-(x)-g^-(y))$, we get 
$$
\iint_{\R^{2n}\setminus(\Omega^c)^2}\frac{|(v-u_-)^-(x)-(v-u_-)^-(y)|^2}{|x-y|^{n+2s}}dx\,dy + \int_\Omega (v-u_-)(v-u_-)^-dx \le 0,$$
that is $(v-u_-)^-\equiv 0$ in $\Omega$. 
In a similar way, we prove that $v\le u^+$ in $\Omega$ (if $u^+<+\infty$).
\end{proof}

\begin{remark} In what follows, we will use indifferently the quantities $\mathcal E'(u)$, $\nabla \mathcal E(u)$ and $u-\tilde{T}(u)$. 
Below, we write explicitly the relations among these three objects.
Given $\mathcal E': H^s_{\Omega,0}\to (H^s_{\Omega,0})^*$, the differential of $\mathcal E$, for every $u\in H^s_{\Omega,0}$, we denote by $\nabla \mathcal E(u)$ the only function of $H^s_{\Omega,0}$ (whose existence is guaranteed by Riesz's Representation Theorem) such that 
$$
(\nabla \mathcal E(u),v)_{H^s_{\Omega,0}} = \mathcal E'(u)[v]\quad\mbox{for all }v\in H^s_{\Omega,0},
$$ 
where $(\cdot,\cdot)_{H^s_{\Omega,0}}$ is the scalar product defined in Section \ref{sec2.0}. In particular, $\|\nabla \mathcal E(u)\|_{H^s_{\Omega,0}}=\|\mathcal E'(u)\|_{*}$, $\|\cdot\|_*$ being the norm in the dual space $(H^s_{\Omega,0})^*$.
Now, by the definition \eqref{eq:tildeT_def} of the operator $\tilde T$, we know that, for every $u\in H^s_{\Omega,0}$, $\tilde T(u)=v$, where $v\in H^s_{\Omega,0}$ is the unique solution of $(-\Delta)^sv+v=\tilde f(u)$ in $\Omega$, under nonlocal Neumann boundary conditions. Therefore, for every $u,\,v\in H^s_{\Omega,0}$, it results
$$
\begin{aligned}
\big(u&-\tilde T(u),v\big)_{H^s_{\Omega,0}}=\\
&=\frac{c_{n,s}}{2}\iint_{\mathbb R^{2n}\setminus(\Omega^c)^2}\frac{(u(x)-u(y))(v(x)-v(y))}{|x-y|^{n+2s}}dx\,dy+\int_{\Omega}uv\,dx\\
&\phantom{=}-\frac{c_{n,s}}{2}\iint_{\mathbb R^{2n}\setminus(\Omega^c)^2}\frac{(\tilde T(u(x))-\tilde T(u(y)))(v(x)-v(y))}{|x-y|^{n+2s}}dx\,dy-\int_{\Omega}\tilde T(u)v\,dx\\
&=\frac{c_{n,s}}{2}\iint_{\mathbb R^{2n}\setminus(\Omega^c)^2}\frac{(u(x)-u(y))(v(x)-v(y))}{|x-y|^{n+2s}}dx\,dy+\int_{\Omega}(u-\tilde f(u))v\,dx\\
&
=\mathcal E'(u)[v].
\end{aligned}
$$
In conclusion, $u-\tilde T(u)=\nabla \mathcal E(u)$ for every $u\in H^s_{\Omega,0}$.
\end{remark}

\begin{lemma}[\textbf{Deformation Lemma in $\mathcal C_*$}]\label{deformation} Let $c\in\mathbb R$ be such that $\mathcal E'(u)\neq 0$ for all $u\in \mathcal C_*$, with $\mathcal E(u)=c$. Then, there exists a function $\eta:\mathcal C_*\to\mathcal C_*$ satisfying the following properties: 
\begin{itemize}
\item[(i)] $\eta$ is continuous with respect to the topology of $H^s_{\Omega,0}$;
\item[(ii)] $\mathcal E(\eta(u))\le \mathcal E(u)$ for all $u\in\mathcal C_*$;
\item[(iii)] $\mathcal E(\eta(u))\le c-\bar\varepsilon$ for all $u\in\mathcal C_*$ such that $|\mathcal E(u)-c|<\bar\varepsilon$;
\item[(iv)] $\eta(u)=u$ for all $u\in\mathcal C_*$ such that $|\mathcal E(u)-c|>2\bar\varepsilon$,
\end{itemize}
where $\bar\varepsilon$ is the positive constant corresponding to $c$ given in Corollary \ref{conseqPS}.
\end{lemma}
\begin{proof} The ideas of this proof are borrowed from \cite[Lemma 4.5]{BNW}, cf. also \cite[Lemma 3.8]{CN}. Let $\chi_1:\mathbb R\to [0,1]$ be a smooth cut-off function such that 
$$\chi_1(t)=
\begin{cases}1\quad&\mbox{if }|t-c|<\bar\varepsilon,\\
0&\mbox{if }|t-c|>2\bar\varepsilon,
\end{cases}
$$
where $\bar\delta$ and $\bar\varepsilon$ are given in Corollary \ref{conseqPS}.
Let $\Phi: H^s_{\Omega,0}\to H^s_{\Omega,0}$ be the map defined by
$$\Phi(u):=\begin{cases}\chi_1(\mathcal E(u))\frac{\nabla \mathcal E(u)}{\|\nabla \mathcal E(u)\|_{H^s_{\Omega,0}}}\quad&\mbox{if }|\mathcal E(u)-c|\le 2\bar\varepsilon,\\
0&\mbox{otherwise.}\end{cases}$$
Note that the definition of $\Phi$ is well posed by Corollary \ref{conseqPS}. 

For all $u\in\mathcal C_*$, we consider the Cauchy problem 
\begin{equation}\label{CauchyProblem}
\begin{cases}\frac{d}{dt}\eta(t,u)=-\Phi(\eta(t,u))\quad t\in(0,\infty),\\
\eta(0,u)=u.
 \end{cases}
\end{equation}
Being $\mathcal E$ of class $C^2$, there exists a unique solution $\eta(\cdot, u)\in C^1([0,\infty);H^s_{\Omega,0})$, cf. \cite[Chapter \$1]{D}. 

We shall prove that for all $t>0$, $\eta(t,\mathcal C_*)\subset \mathcal C_*$. 
Fix $\bar t>0$. For every $u\in\mathcal C_*$ and $k\in\N$ with $k\ge \bar t/\bar\delta$, let
$$\begin{cases}\bar\eta_k(0,u):=u,\\
\bar\eta_k\left(t_{i+1},u\right):=\bar\eta_k\left(t_i,u\right)-\frac{\bar t}k\Phi\left(\bar\eta_k\left(t_i,u\right)\right)\quad\mbox{for all }i=0,\dots,k-1,
\end{cases}
$$
with $$t_i:=i\cdot\frac{\bar t}k\quad\mbox{for all }i=0,\dots,k-1.$$
Let us prove that for all $i=0,\dots,k-1$, $\bar\eta_k\left(t_{i+1},u\right)\in\mathcal C_*$. If $|\mathcal E(u)-c|>2\bar\varepsilon$, then $\bar\eta_k\left(t_{i+1},u\right)=u\in\mathcal C_*$ for every $i=0,\dots,k-1$. Otherwise, let 
$$\lambda:=\frac{\bar t}k\cdot\frac{\chi_1\left(\mathcal E\left(\bar\eta_k\left(t_i,u\right)\right)\right)}{\|\bar\eta_k\left(t_i,u\right)-\tilde T\left(\bar\eta_k\left(t_i,u\right)\right)\|_{H^s_{\Omega,0}}}.$$
Clearly, $\lambda\le1$ by Corollary \ref{conseqPS}, being $k\ge \bar t/\bar\delta$ and $\|u-\tilde T(u)\|_{H^s_{\Omega,0}}=\|\nabla \mathcal E\|_{H^s_{\Omega,0}}$. Therefore,we have for every $i=0,\dots,k-1$
$$\bar\eta_k\left(t_{i+1},u\right)=(1-\lambda)\bar\eta_k\left(t_i,u\right)+\lambda \tilde T\left(\bar\eta_k\left(t_i,u\right)\right)\in \mathcal C_*$$
by induction on $i$, and by the convexity of $\mathcal C_*$. 
For every $i=0,\dots,k-1$, we can now define the line segment
$$
\eta_k^{(i)}(t,u):= \left(1-\frac{t}{\bar t} k+i\right)\bar\eta_k\left(t_i,u\right)+\left(\frac{t}{\bar t} k-i\right)\bar\eta_k\left(t_{i+1},u\right)
$$
for all $t\in\left[t_i,t_{i+1}\right]$.
We denote by $\eta_k:=\bigcup_{i=0}^{k-1}\eta^{(i)}_k$ the whole Euler polygonal defined in $[0,\bar t ]$. Being $\mathcal C_*$ convex, we get immediately that for all $t\in[0,\bar t ]$, $\eta_k(t,u)\in\mathcal C_*$. 

We claim that $\eta_k(\cdot,u)$ converges to the solution $\eta(\cdot,u)$ of the Cauchy problem \eqref{CauchyProblem} in $H^s_{\Omega,0}$.
Indeed, for all $i=0,\dots,k-1$, we integrate by parts the equation of \eqref{CauchyProblem} in the interval $[t_i,t_{i+1}]$ and we obtain
$$\eta(t_{i+1},u)=\eta(t_i,u)-\frac{\bar t}k\Phi(\eta(t_i,u))+\int_{t_i}^{t_{i+1}}(\tau-t_{i+1})\frac{d}{d\tau}\Phi(\eta(\tau,u))d\tau.$$
On the other hand, we define the error
$$\varepsilon_i:=\|\eta(t_i,u)-\eta_k(t_{i},u)\|_{H^s_{\Omega,0}}\quad\mbox{for every }i=0,\dots,k-1.$$
Hence, for every $i=0,\dots,k-1$, we get 
\begin{equation}\label{error-estimate}
\begin{aligned}
\varepsilon_{i+1}\le\varepsilon_i+&\frac{\bar t}k\|\Phi(\eta(t_i,u))-\Phi(\eta_k(t_i,u))\|_{H^s_{\Omega,0}}\\
&+\left\|\int_{t_i}^{t_{i+1}}(t_{i+1}-\tau)\frac{d}{d\tau}\Phi(\eta(\tau,u))d\tau\right\|_{H^s_{\Omega,0}}.
\end{aligned}
\end{equation}
Now, since $\Phi$ is locally Lipschitz and $\eta([0,\bar t])\subset H^s_{\Omega,0}$ is compact, 
\begin{equation}\label{ineqonPhi}
\|\Phi(\eta(t_i,u))-\Phi(\eta_k(t_i,u))\|_{H^s_{\Omega,0}}\le\varepsilon_i L_\Phi
\end{equation}
for some $L_\Phi=L_\Phi(\eta([0,\bar t]))>0$. 
Furthermore, 
$$\begin{aligned}\left\|\int_{t_i}^{t_{i+1}}(t_{i+1}-\tau)\frac{d}{d\tau}\Phi(\eta(\tau,u))d\tau\right\|_{H^s_{\Omega,0}}&\le\int_{t_i}^{t_{i+1}}(t_{i+1}-\tau)\left\|\frac{d}{d\tau}\Phi(\eta(\tau,u))\right\|_{H^s_{\Omega,0}}d\tau\\
&\le\frac{\bar t}{k}\int_0^{\bar t}\|\Phi'(\eta(\tau,u))\|_*\|\Phi(\eta(\tau,u))\|_{H^s_{\Omega,0}}d\tau\\
&\le\frac{\bar t^2}{k}\sup_{\tau\in[0,\bar t]}\|\Phi'(\eta(\tau,u))\|_*= \frac{\bar t^2}{k}L_\Phi.
\end{aligned}$$
Thus, combining the last inequality with \eqref{ineqonPhi} and \eqref{error-estimate}, we have
$$\varepsilon_{i+1}\le\varepsilon_i+\frac{\bar t}k\varepsilon_i L_\Phi+\frac{\bar t^2}{k}L_\Phi\quad\mbox{for all }i=0,\dots,k-1.$$
This implies that 
$$\varepsilon_{i+1}\le \frac{\bar t^2}k L_\Phi\sum_{j=0}^i\left(1+\frac{\bar t}k L_\Phi\right)^j=\bar t\left[\left(1+\frac{\bar t}k L_\Phi\right)^{i+1}-1\right]\to 0\quad\mbox{as }k\to\infty,$$
where we have used the fact that $\varepsilon_0=0$. By the triangle inequality and the continuity of $\eta(\cdot,u)$ and $\eta_k(\cdot,u)$, this yields the claim. 

Hence, for all $t\in[0,\bar t]$, $\eta(t,u)\in\mathcal C_*$ by the closedness of $\mathcal C_*$.   

For all $u\in\mathcal C_*$ and $t>0$ we can write
\begin{equation}\label{eq:flusso_decrescente}
\begin{aligned}\mathcal E(\eta(t,u))-\mathcal E(u)&=\int_0^t\frac{d}{d\tau}\mathcal E(\eta(\tau,u))d\tau\\
&\hspace{-2.5cm}=-\int_0^t\frac{\chi_1(\mathcal E(\eta(\tau,u)))}{\|\eta(\tau,u)-\tilde T(\eta(\tau,u))\|_{H^s_{\Omega,0}}}\mathcal E'(\eta(\tau,u))[\eta(\tau,u)-\tilde T(\eta(\tau,u))]d\tau\\
&\hspace{-2.5cm}
=-\displaystyle{\int_0^t \|\eta(\tau,u)-\tilde T(\eta(\tau,u))\|_{H^s_{\Omega,0}}\chi_1(\mathcal E(\eta(\tau,u)))d\tau}\le0.
\end{aligned}
\end{equation}

Now, let $u\in\mathcal C_*$ be such that $|\mathcal E(u)-c|<\bar\varepsilon$ and let $t\ge 2\bar\varepsilon/\bar\delta$. Then, two cases arise: either there exists $\tau\in[0,t]$ for which $\mathcal E(\eta(\tau,u))\le c-\bar\varepsilon$ and so, by the previous calculation we get immediately that $\mathcal E(\eta(t,u))\le c-\bar\varepsilon$, or for all $\tau\in[0,t]$, $\mathcal E(\eta(\tau,u))> c-\bar\varepsilon$. In this second case, 
$$c-\bar\varepsilon< \mathcal E(\eta(\tau,u))\le \mathcal E(u)< c+\bar\varepsilon.$$
In particular, by the definition of $\chi_1$, and by Corollary~\ref{conseqPS}, we have that for all $\tau\in[0,t]$ 
$$\chi_1(\mathcal E(\eta(\tau,u)))=1,\qquad\|\eta(\tau,u)-\tilde T(\eta(\tau,u))\|_{H^s_{\Omega,0}}\ge\bar\delta.$$
Hence, by \eqref{eq:flusso_decrescente}, we obtain
$$\mathcal E(\eta(t,u))\le
\mathcal E(u)-\displaystyle\int_0^t\bar\delta d\tau 
\le c+\bar\varepsilon-\bar\delta t\le c-\bar\varepsilon.
$$

Finally, if we define with abuse of notation $$\eta(u):=\eta\left(\frac{2\bar\varepsilon}{\bar\delta},u\right),$$ it is immediate to verify that $\eta$ satisfies (i)-(iv).
\end{proof}

\begin{lemma}[\textbf{Mountain pass geometry}]
\label{geometry}
Let $\tau>0$ be such that $\tau <\min \{u_0-u_-,u_+-u_0\}$. 
Then there exists $\alpha>0$ such that 
\begin{itemize}
\item[(i)] $\mathcal E(u)\ge \mathcal E(u_-)+\alpha$ for every $u \in
\mathcal C_*$ with $\|u-u_-\|_{L^\infty(\Omega)}=\tau$;
\item[(ii)] if $u_+< \infty$, then $\mathcal E(u)\ge \mathcal E(u_+)+\alpha$ for every $u \in
\mathcal C_*$ with $\|u-u_+\|_{L^\infty(\Omega)}= \tau$;
\item[(iii)] if $u_+=+\infty$, then there exists $\bar u\in\mathcal C_*$ with $\|\bar u-u_-\|_{L^\infty(\Omega)}>\tau$ such that $\mathcal E(\bar u)< \mathcal E(u_-)$. 
\end{itemize}
\end{lemma}
\begin{proof} The proof is analogous to the one of \cite[Lemma 4.6]{BNW}, we report it here for the sake of completeness. 
Suppose by contradiction that there exists a sequence $(w_k)
\subset \mathcal C_*$ such that
\begin{equation}\label{wnbounded}
\|w_k\|_{L^\infty(\Omega)}=w_k(R)=\tau>0\quad\mbox{for all }k
\end{equation}
and $\limsup \limits_{k \to \infty}
\bigl[\mathcal E(u_-+w_k)-\mathcal E(u_-)\bigr] \le 0$. Since
$$
\begin{aligned}
&\frac12\int_\Omega((u_-+w_k)^2-u_-^2)dx=\int_\Omega\int_0^1(u_-+t w_k)w_k\,dtdx,\\
&\tilde F(u_-+w_k)-\tilde F(u_-)=\int_0^1\tilde f(u_-+tw_k)w_k dt,
\end{aligned}
$$
we get
\begin{align*}
\mathcal E&(u_-+w_k)-\mathcal E(u_-)\\ 
&= \frac{1}2\left([w_k]_{H^s_{\Omega,0}}^2+ \int_\Omega[(u_-+w_k)^2 -u_-^2]dx \right) 
 - \int_\Omega \Bigl(\tilde F(u_- + w_k)-\tilde F(u_-))\,dx\\
&= \frac{1}2\left([ w_k ]_{H^s_{\Omega,0}}^2 + 
\int_\Omega \int_0^1 \Bigl(u_-+t w_k - \tilde f(u_-+t w_k)\Bigr)w_k\,dt dx\right). 
\end{align*} 
Therefore, since by $(f_3)$ and the definition of $u_-$ 
\begin{equation}\label{fmp}
t-\tilde f(t)>0 \qquad \text{for $t \in (u_-,u_0)$,}  
\end{equation}
we conclude that $[w_k]_{H^s_{\Omega,0}} \to 0$. We claim that $(w_k)$ converges to the constant solution $w\equiv \tau$ in the $H^s_{\Omega,0}$ norm. Indeed, using $[w_k]_{H^s_{\Omega,0}} \to 0$ and \eqref{wnbounded}, we have that $(w_k)$ is bounded in $H^s_{\Omega,0}$ and so, up to a subsequence, it weakly converges to some $w\in H^s_{\Omega,0}$. Hence, 
\begin{equation}\label{fromwc}
\begin{aligned}
0&=\lim_{k\to\infty}\iint_{\mathbb R^{2n}\setminus (\Omega^c)^2} \frac{[(w_k-w)(x)-(w_k-w)(y)](w(x)-w(y))}{|x-y|^{n+2s}}dx dy\\
&=\lim_{k\to\infty}\iint_{\mathbb R^{2n}\setminus (\Omega^c)^2} \frac{(w_k(x)-w_k(y))(w(x)-w(y))}{|x-y|^{n+2s}}dx dy -[w]^2_{H^s_{\Omega,0}}.
\end{aligned}
\end{equation}
Moreover, 
\begin{equation}\label{CauchySchwartz}
\frac{c_{n,s}}{2}\iint_{\mathbb R^{2n}\setminus (\Omega^c)^2} \frac{(w_k(x)-w_k(y))(w(x)-w(y))}{|x-y|^{n+2s}}dx dy\le C [w_k]_{H^s_{\Omega,0}}[w]_{H^s_{\Omega,0}}.
\end{equation}
Combining \eqref{fromwc} and \eqref{CauchySchwartz}, we get $[w]_{H^s_{\Omega,0}}=0$, which implies that $w\equiv \tau$.  Thus, $(w_k)$ converges to the constant $\tau$ in $H^s_{\Omega,0}$.
By the Dominated Convergence Theorem we can conclude that 
\begin{align*}
0 &\ge \lim_{k \to \infty} \int_\Omega \int_0^1 \Bigl(u_-+t w_k - \tilde f(u_-+t w_k)\Bigr)w_k\,dtdx\\
&=  \int_\Omega \int_0^1 \Bigl(u_-+t \tau - \tilde f(u_-+t\tau)\Bigr)\tau\,dt dx,
\end{align*}
which contradicts \eqref{fmp}. Hence there exists $\alpha_1>0$ such that (i) holds. 

In a similar way, now using the fact that $t-\tilde f(t)<0$ for $t \in (u_0,u_+)$, we find $\alpha_{2}>0$
such that (ii) holds if $u_+< \infty$. The claim then follows with $\alpha:= \min \{\alpha_1,\alpha_2\}$.

Finally, if $u_+=+\infty$, the existence of a point $\bar u\in \mathcal C_*$ outside the crest centered in $u_-$ is guaranteed  by the following estimate (cf. also \cite[Remarks p. 118]{C-Bruto}):
\begin{equation}\label{I-infty}\begin{aligned}\mathcal E(t\cdot 1)&=|\Omega|\left(\frac{t^2}2-\int_0^t\tilde f(s)ds\right)\\
&\le |\Omega|\left(\frac{t^2}2-\int_0^M\tilde f(s)ds-(1+\delta)\int_M^ts ds\right)\\
&\le\frac{|\Omega|}2\left(t^2-2M\min_{s\in[0,M]}\tilde f(s)-(1+\delta)(t^2-M^2)\right)\\
&=C-\frac{|\Omega|\delta}2t^2\to-\infty\quad\mbox{as }t\to\infty,
\end{aligned}
\end{equation}
where we have used the fact that $\tilde f\in\mathfrak F_{M,\delta}$. This shows (iii) and concludes the proof.
\end{proof}

\begin{remark}
We observe that, comparing (i) and (ii) in Lemma \ref{geometry}, it is apparent that, whenever $u_+<+\infty$, if $\mathcal E(u_-)<\mathcal E(u_+)$, then $u_+$ plays the role of the center inside the crest of the mountain pass and $u_-$ plays the role of the point outside the crest with less energy, otherwise the roles of $u_-$ and $u_+$ have to be interchanged. 
\end{remark}

Now, let
\begin{equation}\label{eq:2}
\begin{aligned} U_- &:= \left\{u \in \mathcal C_* \::\: \mathcal E(u)<\mathcal E(u_-)+\frac{\alpha}{2},\:
\|u-u_- \|_{L^\infty(\Omega)} < \tau\right\},\\
\\
U_+&:=\begin{cases}
\displaystyle{\left\{u \in \mathcal C_* \::\: \mathcal E(u)<\mathcal E(u_+)+\frac{\alpha}{2},\:
\|u-u_+ \|_{L^\infty(\Omega)} < \tau\right\}},&\mbox{ if }u_+<\infty,\\
&\\
\left\{u \in \mathcal C_* \, :\, \mathcal E(u)< \mathcal E(u_-),\, \|u-u_-\|_{L^\infty(\Omega)}>\tau\right\},&\mbox{ if }u_+=\infty
\end{cases}
\end{aligned}
\end{equation}
where $\tau$ and $\alpha$ are given by Lemma \ref{geometry},

$$
\Gamma:=\left\{ \gamma\in C([0,1];\mathcal C_*)\ :\  \gamma(0) \in U_-,\:
  \gamma(1) \in U_+\right\},
$$
and
\begin{equation}\label{minmax}
c:=\inf_{\gamma\in\Gamma}\max_{t\in[0,1]} \mathcal E(\gamma(t)).
\end{equation}

\begin{remark}
The reason for considering two sets, $U_+$ and $U_-$, instead of just two points for the starting and the ending points of the admissible curves will be clear in Lemma \ref{lemma:nonconstant_p>2}. Indeed, this choice makes easier exhibiting an admissible curve along which the energy is lower than the energy of the constant.  
\end{remark}

\begin{proposition}[\textbf{Mountain Pass Theorem}]\label{mountainpass} The value $c$ defined in \eqref{minmax} is finite and there exists a
critical point $u\in\mathcal C_*\setminus\{u_-,u_+\}$ of $\mathcal E$ with $\mathcal E(u)=c$. In particular, $u$ is a weak solution of \eqref{P}.
\end{proposition}
\noindent The proof of the above proposition is standard, once one has the mountain pass geometry (Lemma \ref{geometry}) and the deformation Lemma (Lemma \ref{deformation}). We refer e.g. to \cite[Proposition 3.10]{CN} for a proof given in a very similar situation.

\section{Non-constancy of the minimax solution}\label{sec4}
In this section we prove that the solution $u\in\mathcal C_*$, whose existence has been established in the previous section, is non-constant. Since we work in the restricted cone $\mathcal C_*$ where the only constant solutions are $u_-$, $u_+$, and $u_0$, and since the mountain pass geometry guarantees that $u\not\equiv u_-$ and $u\not\equiv u_+$ (cf. Proposition \ref{mountainpass}), it is enough to prove that $u\not\equiv u_0$. To this aim, following the idea in \cite[Section 4]{BNW}, we first prove that on the Nehari-type set 
$$
 N_*:=\{u\in\mathcal C_*\setminus\{0\}\,:\, \mathcal E'(u)[u]=0\},
$$
i.e., roughly speaking, on the crest of the mountain pass, the infimum of the energy is strictly less than $\mathcal E(u_0)$, cf. also \cite[Remark 2]{CN}. Then, we explicitly build an admissible curve $\bar \gamma\in \Gamma$ along which the energy is less than $\mathcal E(u_0)$. By \eqref{minmax}, this ensures that the mountain pass level is less than $\mathcal E(u_0)$ and so $u\not\equiv u_0$.\smallskip

We start by introducing some useful notation. 
We denote by 
$$
H^s_\mathrm{rad}:=\{u\in H^s_{\Omega,0}\,:\, u\mbox{ radial }\},
$$
we introduce also the space of radial, non-decreasing functions
$$
H^s_\mathrm{+,r}:=\{u\in H^s_{\Omega,0}\,:\, u\mbox{ radial and radially non-decreasing }\}.
$$
We define the second radial eigenvalue $\lambda_2^\mathrm{rad}$ and the second radial increasing eigenvalue $\lambda_2^{+,\mathrm{r}}$ of the fractional Neumann Laplacian in $\Omega$ as follows:
\begin{equation}\label{lambda2rad+}
\lambda_2^\mathrm{rad}:=\inf_{v\in H^s_\mathrm{rad},\,\int v=0}\frac{[v]^2_{H^s_{\Omega,0}}}{\int_{\Omega}v^2},\qquad\quad
\lambda_2^\mathrm{+,r}:=\inf_{v\in H^s_\mathrm{+,r},\,  \int v=0}\frac{[v]^2_{H^s_{\Omega,0}}}{\int_{\Omega}v^2}.
\end{equation}
Clearly, the following chain of inequalities holds by inclusion $H^s_{+,\mathrm{r}}\subset H^s_\mathrm{rad}\subset H^s_{\Omega,0}$
$$
0<\lambda_2\le \lambda_2^\mathrm{rad}\le \lambda_2^\mathrm{+,r}
$$
and, by the direct method of Calculus of Variations, all these infima are achieved.

\begin{remark}
We observe that in the local case, i.e., for the Neumann Laplacian, it is known that the second radial eigenfunction is increasing, so that the second radial eigenvalue and the second radial {\it increasing} eigenvalue coincide. In this nonlocal setting we do not know whether the same equality holds true. 
In \cite{BNW}, for the local case, the condition required on $f'(u_0)$ involves the second radial eigenvalue, and the proof of the non-constancy of the solution uses the monotonicity of the associated eigenfunction. In this paper, we need to require an assumption involving $\lambda_2^{+,\mathrm{r}}$, which, as explained above, might be more restrictive. On the other hand, as will be clear in Proposition \ref{constant-sol}, some condition on the derivative of $f$ is needed in order to guarantee the existence of non-constant solutions. 
\end{remark}
 
\begin{lemma}\label{4.9}
Let $v_2\in H^s_{+,\mathrm{r}}$ be the second radial increasing eigenfunction, namely the function that realizes $\lambda_2^{+,\mathrm{r}}$. 
Let
$$
\psi: \R^2 \to \R,\qquad \psi(s,t):=\mathcal E'(t(u_0+sv_2))[u_0+sv_2], 
$$
then there exist $\eps_1,\eps_2>0$ and a
$C^1$ function $h:(-\eps_1,\eps_1) \to (1-\eps_2,1+\eps_2)$
such that for $(s,t) \in V:= (-\eps_1,\eps_1) \times
(1-\eps_2,1+\eps_2)$ we have 
\begin{equation}\label{psi=0}
\psi(s,t)=0\quad \mbox{if and only if}\quad t=h(s).
\end{equation}
Moreover, 
\begin{itemize}
\item[(i)] $h(0)=1$, $h'(0)=0$;\smallskip
\item[(ii)] $\frac{\partial}{\partial t}\psi(s,t)<0$ for $(s,t)\in V$;\smallskip
\item[(iii)] $\mathcal E(h(s)(u_0+sv_2))<\mathcal E(u_0)$ for $s \in (-\eps_1,\eps_1)$, $s\neq 0$.
\end{itemize}
The same result holds true replacing $v_2$ with the second radial decreasing eigenfunction $-v_2$ (which clearly corresponds to the same eigenvalue $\lambda_2^{+,\mathrm{r}}$).
\end{lemma}
\begin{proof} The proof is similar to the one of \cite[Lemma 4.9]{BNW}, we report it here because it highlights the importance of assumption $(f_3)$. Part (i) follows by the Implicit Function Theorem applied to $\psi$.
Indeed, since $\mathcal E$ is a $C^2$ functional and $\psi$ is of class $C^1$ with $\psi(0,1)=0$, by $(f_3)$ we get 
\begin{equation}\label{psi'<0}
\frac{\partial}{\partial t}\Big|_{(0,1)}\psi(s,t) = \mathcal E''(u_0)[u_0,u_0]= [1-\tilde f'(u_0)]\int_{B} 
u_0^2 \,dx <0,
\end{equation}
where we have used only that $\tilde{f}'(u_0)=f'(u_0)>1$. Furthermore, since $\int_\Omega v_2=0$,
\begin{equation}\label{psi'=0}\begin{aligned}
\frac{\partial}{\partial s}\Big|_{(0,1)}\psi(s,t) &=\mathcal E'(u_0)[v_2]+ \mathcal E''(u_0)[u_0,v_2]\\
&=
[1-\tilde f'(u_0)]u_0 \int_\Omega 
v_2\,dx =0.
\end{aligned}
\end{equation}
Thus, the Implicit Function Theorem guarantees the existence of $\eps_1,\eps_2$ and $h$, as well as property (i).
Then, part (ii) is a consequence of the regularity of $\psi$. We prove now (iii), here is where $(f_3)$ plays a crucial role. 
By (i), we can write $h(s)= 1+o(s)$, for $s\in(-\varepsilon_1,\varepsilon_1)$, $s\neq 0$, so that 
$$
h(s)(u_0+sv_2)-u_0= sv_2 + o(s)
$$
and therefore, by Taylor expansion and $(f_3)$,
\begin{align*}
\mathcal E(h(s)(u_0+sv_2)) -\mathcal E(u_0) &= \frac{1}{2} \mathcal E''(u_0)[sv_2 +
o(s),sv_2+o(s)]+o(s^2)\\&=\frac{s^2}{2} \mathcal E''(u_0)[v_2,v_2]+o(s^2)\\
&=\frac{s^2}{2}\left([v_2]_{H^s_{\Omega,0}}^2+\int_\Omega[1-\tilde
  f'(u_0)]v_2^2\,dx\right)+o(s^2)\\
  &<\frac{s^2}{2}\left([v_2]_{H^s_{\Omega,0}}^2-\lambda_2^{+,\mathrm{r}}\int_\Omega v_2^2\,dx\right)+o(s^2).
\end{align*}
Then, being 
$$
[v_2]_{H^s_{\Omega,0}}^2-\lambda_2^{+,\mathrm{r}}\int_\Omega v_2^2dx=0,
$$
property (iii) holds taking $\eps_1$, $\eps_2$ smaller if necessary. 
\end{proof}

In the following lemma, we build a curve $\gamma_{\bar \tau}$ along which the energy is always less than $\mathcal E(u_0)$. The admissible curve $\bar \gamma\in \Gamma$ with the same property will be a simple reparametrization of $\gamma_{\bar \tau}$.

\begin{lemma}\label{lemma:nonconstant_p>2}
Fix 
$0<t_-<1<t_+$ such that  
\begin{equation}
  \label{eq:3}
t_- u_0 \in U_-,\quad t_+
u_0 \in U_+ \quad \text{and}\quad    u_- < t_- u_0 < u_0 < t_+ u_0 < u_+, 
\end{equation}
where $U_\pm$ are defined in $(\ref{eq:2})$. 
Let $v_2$ be the second radial increasing eigenfunction as in Lemma \ref{4.9}.
For $\tau  \ge 0$ define
\begin{equation}
  \label{eq:4}
  \begin{aligned}
\gamma_\tau: [t_-,t_+]\to H^s_{\Omega,0} \qquad &\gamma_\tau(t):= t(u_0+\tau v_2)\\ &(\mbox{resp. }\gamma_\tau(t):= t(u_0-\tau v_2)).
\end{aligned}
\end{equation}
Then there exists $\bar \tau>0$ such that $\gamma_{\bar \tau}(t_\pm ) \in U_\pm$, $\gamma_{\bar \tau}(t) \in \mathcal C_{+,*}$ (resp. $\mathcal C_{-,*}$)  for $t_- \le t \le t_+$ and 
\begin{equation}
  \label{eq:1}
\max_{t_- \le t \le t_+} \mathcal E(\gamma_{\bar \tau}(t))<\mathcal E(u_0).  
\end{equation}
As a consequence, there exists an admissible curve $\bar\gamma\in\Gamma$ along which the energy is always lower than $\mathcal E(u_0)$.
\end{lemma}
\noindent For the proof of the previous lemma, we refer to \cite[Lemma 4.10]{BNW}, see also \cite[Lemma~4.2]{CN}. Here the monotonicity of $v_2$ (resp. of $-v_2$) is essential to guarantee that  $\gamma_{\bar\tau}([t_-,t_+])\subset\mathcal C_{+,*}$ (resp. $\mathcal C_{-,*}$). Finally, the admissible curve $\gamma\in\Gamma$ is given in terms of $\gamma_{\bar\tau}$ as follows  
$$\bar\gamma(t):=\gamma_{\bar \tau}(t(t_+-t_-)+t_-)\quad\mbox{ for all }t\in[0,1].$$ 

\begin{proof}[$\bullet$ Proof of Theorem \ref{thm:main}]
By Proposition \ref{mountainpass}, there exists a mountain pass type solution $u\in\mathcal C_*\setminus\{u_-,\,u_+\}$ of \eqref{P} such that $\mathcal E(u)=c$. Moreover, by Lemma \ref{lemma:nonconstant_p>2} and the definition of the minimax level $c$ given in \eqref{minmax}, we have that
$$
c\le\max_{t\in [0,1]}\mathcal E(\bar\gamma(t))<\mathcal E(u_0),
$$
that is $u\not\equiv u_0$, and so $u$ is non-constant.
 Furthermore, $u>0$ a.e. in $\Omega$ by the maximum principle stated in Theorem \ref{MaxPrinc} combined with the regularity of $u$ given in Lemma \ref{regularity}. Actually, since $u$ is smooth and non-decreasing, $u>0$ in $\Omega\setminus\{0\}$. 
 
The multiplicity part of the statement is proved by reasoning in the same way for each $u_{0,i}$, with $i=1,\dots,N$. Indeed, assume without loss of generality that $u_{0,1}<u_{0,2}<\dots<u_{0,N}$. For every $i$, we define $u_{\pm,i}$ and the cone of non-negative, radial, non-decreasing (or non-increasing) functions 
$\mathcal C_{*,i}$, corresponding to $u_{0,i}$. Then 
\begin{equation}\label{order}
u_{-,1}<u_{+,1}\le u_{-,2}<\dots\le u_{+,N}.
\end{equation} 
Proceeding as in the present and in the previous sections, for every $i$, we get a non-constant positive solution $u_i\in\mathcal C_{*,i}$. Hence, by \eqref{order}, 
$$
u_{-,1}\underset{\not\equiv}{\le} u_1\underset{\not\equiv}{\le} u_{+,1}\underset{\not\equiv}{\le}
u_{-,2}\underset{\not\equiv}{\le} u_2\underset{\not\equiv}{\le} u_{+,2}\underset{\not\equiv}{\le}\dots\underset{\not\equiv}{\le} u_N\underset{\not\equiv}{\le} u_{+,N},
$$
which proves in particular that the $N$ solutions are distinct.
\end{proof}

The following proposition gives a sufficient condition on $f$ under which problem \eqref{P} admits only constant solutions. We recall that $K_\infty$ denotes the uniform bound on the $L^\infty$ norm of $u$ given in Lemma \ref{L-infty-unif}.

\begin{proposition}\label{constant-sol}
Let $\delta \in (0, \lambda_2^{\mathrm{rad}})$ and $M>0$. Suppose that $f\in\mathfrak F_{M,\delta}$ satisfies $(f_1)$ and $(f_2)$. If $f'(t)<\lambda_2^{\mathrm{rad}}+1$ for every $t\in [0,K_\infty]$, then problem \eqref{P} admits only constant solutions in $H^s_{\mathrm{rad}}$. 
\end{proposition}
\begin{proof}
We first observe that, if $M<K_\infty$, condition $f'<\lambda_2^{\mathrm{rad}}+1$ in $[0,K_\infty]$ is compatible with the consequence \eqref{Mdelta} of $(f_2)$, when $\delta<\lambda_2^{\mathrm{rad}}$.
Let $u\in H^s_{\mathrm{rad}}$ be a weak solution of \eqref{P}. We can write $u=v+\mu$ for some $\mu \in \R$ and $v\in H^s_{\mathrm{rad}}$ with 
$$\begin{aligned}&\int_\Omega v \, dx=0\quad \mbox{and}\\
&\lambda_2^{\mathrm{rad}}\int_\Omega v^2\,dx \leq \frac{c_{n,s}}{2}\iint_{\R^{2n}\setminus(\Omega^c)^2}\frac{|v(x)-v(y)|^2}{|x-y|^{n+2s}}\,dx\,dy + \int_\Omega v^2\,dx.
\end{aligned}
$$
Using the definition of weak solution for $u=v+\mu$ and testing with $v$, we get
\[\begin{split}
& (\lambda_2^{\mathrm{rad}}+1) \int_\Omega v^2\,dx \leq \frac{c_{n,s}}{2}\iint_{\R^{2n}\setminus (\Omega^c)^2}\frac{|v(x)-v(y)|^2}{|x-y|^{n+2s}}\,dx\,dy + \int_\Omega v^2\,dx\\
&\hspace{1em} =\int_\Omega f(v+\mu)v\,dx=\int_\Omega [f(v+\mu)-f(\mu)]v\,dx=\int_\Omega f'(\mu + \omega v) v^2\,dx,
\end{split}
\]
where $\omega=\omega(x)$ satisfies $0\le \omega\le 1$ in $\Omega$. Using that $\|u\|_{L^\infty(\Omega)}\le K_\infty$, we deduce that $\|\mu + \omega v\|_{L^\infty(\Omega)}\le K_\infty$. Therefore, since by assumption $f'(\mu + \omega v)< \lambda_2^{\mathrm{rad}}+1$, we conclude that it must be $v=0$ and thus $u$ identically constant.
\end{proof}

\begin{remark}
Some further comments on the condition $(f_3)$ and its variants are now in order. In the local setting, it was first conjectured in \cite{BNW} and then proved in \cite{BGT,MCL,BCN} that if $f'(u_0)$ satisfies
\begin{equation}\label{f'u0-k}
f'(u_0)>1+\lambda_{k+1}^\mathrm{rad}(R)\quad\mbox{for some }k\ge1,
\end{equation} 
where $\lambda_{k+1}^\mathrm{rad}(R)$ is the $(k+1)$-st radial eigenvalue of the Neumann Laplacian in $B_R$, then the Neumann problem $-\Delta u+u=f(u)$ in $B_R$ admits a radial positive solution having exactly $k$ intersections with the constant $u_0$. It would be interesting to prove a similar result also in this fractional setting. It is worth stressing that the solution $u$ that we find in the present paper is morally the one with one intersection with $u_0$. This is due to the monotonicity of $u\in \mathcal C_*$, the identity holding for solutions of \eqref{P}  
\begin{equation}\label{u=fu}
\int_{\Omega}udx=\int_\Omega f(u)dx,
\end{equation}
and the fact that $f(t)<t$ for $t\in(u_-,u_0)$ and $f(t)>t$ in $(u_0,u_+)$, cf. \eqref{u-+} and \eqref{Cstar}.
 
We conclude this remark observing that, since $\lambda_{k}^\mathrm{rad}(R)\to0$ as $R\to\infty$, condition \eqref{f'u0-k} can be also read as a condition on the size of the domain $B_R$. 
\end{remark}

\section*{Acknowledgments} This work was partially supported by Gruppo Nazionale per l'Analisi Matematica, la Probabilit\`a e le loro Applicazioni (GNAMPA) of the I\-stituto Nazionale di Alta Matematica (INdAM).
The authors are grateful to Proff. Lorenzo Brasco and Benedetta Noris for useful discussions on the subject.
The authors acknowledge the supports of both Departments of Mathematics of the Universities of Bologna and of Turin for their visits in Bologna and in Turin, during which parts of this work have been achieved.  
E. C. is supported by MINECO grants MTM2014-52402-C3-1-P and MTM2017-84214-C2-1-P,
and is part of the Catalan research group 2014 SGR 1083.
F. C. is partially supported by the INdAM-GNAMPA Project 2019 ``Il modello di Born-Infeld per l'elettromagnetismo nonlineare: esistenza, regolarit\`a e molteplicit\`a di soluzioni'' and the project of the University of Turin Ricerca Locale 2018 Linea B (CDD 09/07/2018)- ``Problemi non lineari'' COLF\_RILO\_18\_01. 
\noindent  

\bibliographystyle{abbrv}
\bibliography{biblio}

\end{document}